\documentclass[a4paper,12pt]{article}

\usepackage{amssymb, manfnt}
\usepackage{hyperref}
\usepackage[normalem]{ulem}
\usepackage{tikz}
\usepackage{booktabs}
\usepackage{multirow}
\usepackage{MnSymbol}
\usepackage{graphicx}
\usepackage{tikz-cd}
\usepackage{longtable}
\usepackage{amsmath,amscd,amsthm}

\usepackage{tkz-graph}
\tikzstyle{vertex}=[circle, draw, inner sep=0pt, minimum size=6pt]

\usepackage[all]{xy}
\graphicspath{ {graphs/} }
\usetikzlibrary{decorations.text}
\usepackage{graphicx}
\usepackage[outdir=./pictures]{epstopdf}
\graphicspath{ {pictures/} }
\usepackage[all]{xy}
\usepackage{adjustbox}
\usepackage{graphicx}
\usepackage{fullpage}
\usepackage[T1]{fontenc}
\usepackage{ifthen}
\usepackage[all]{xy}
\usepackage{pdflscape}
\usepackage{array}
\usepackage{multirow}
\usepackage{rotating}

\def\altdb{\vadjust{\vbox to 0pt{\vss\hbox{\kern \hsize
\quad{\dbend}}\kern\baselineskip\kern-10pt}}}

\newcommand{\arxiv}[1]{\href{http://arxiv.org/abs/#1}{\tt arXiv:\nolinkurl{#1}}}
\newcommand{\doi}[1]{\href{http://dx.doi.org/#1}{{\tt DOI:#1}}}
\newcommand{\euclid}[1]{\href{http://projecteuclid.org/getRecord?id=#1}{{\tt #1}}}
\newcommand{\mathscinet}[1]{\href{http://www.ams.org/mathscinet-getitem?mr=#1}{\tt #1}}
\newcommand{\googlebooks}[1]{(preview at \href{http://books.google.com/books?id=#1}{google books})}

\makeatletter
\let\@@pmod\pmod
\DeclareRobustCommand{\pmod}{\@ifstar\@pmods\@@pmod}
\def\@pmods#1{\mkern4mu({\operator@font mod}\mkern 6mu#1)}
\makeatother

\newcommand\id{\operatorname{id}}

\newcommand\ad{\operatorname{Ad}}

\newcommand\Z[1]{ \mathbb{Z}_{#1}}

\newcommand\cC{ \mathcal{C}}

\theoremstyle{plain}
\newtheorem{theorem}{Theorem}[section]
\newtheorem*{theorem*}{Theorem}
\newtheorem*{prop*}{Proposition}
\newtheorem{cor}[theorem]{Corollary}
\newtheorem{lemma}[theorem]{Lemma}

\newtheorem{rmk}[theorem]{Remark}
\theoremstyle{remark}

\theoremstyle{definition}
\newtheorem{example}[theorem]{Example}

\theoremstyle{definition}
\newtheorem{dfn}[theorem]{Definition}

\newcommand{\change}[1]{#1}
\newcommand{\changed}[1]{#1}
\newcommand{\changedd}[1]{#1}

\usepackage{color}

\numberwithin{equation}{section}

\DeclareRobustCommand*{\nicefrac}{\@UnitsNiceFrac}%
\makeatother

\newcounter{sublOne}
\newenvironment{sublOne}[1][]{\refstepcounter{sublOne}\par\medskip
   \noindent \textbf{Sublemma~\ref{lem:hard}.\thesublOne. #1} \rmfamily}{\medskip}

\usepackage{authblk}

\usepackage{longtable}
\title{A complete classification of \change{unitary} fusion categories tensor generated by an object of dimension $\frac{1 + \sqrt{5}}{2}$}

\author{Cain Edie-Michell}

\date{}

\allowdisplaybreaks
\begin{document}

\maketitle

\begin{abstract}
In this paper we give a complete classification of \change{unitary} fusion categories $\otimes$-generated by an object of dimension $\frac{1 + \sqrt{5}}{2}$. We show that all such categories arise as certain wreath products of either the Fibonacci category, or of the dual even part of the $2D2$ subfactor. As a by-product of proving our main classification result we produce a classification of finite \changedd{unitarizable} quotients of $\operatorname{Fib}^{*N}$ satisfying a certain symmetry condition.
\end{abstract}

\section{Introduction}\label{sec:intro}

Fusion categories are rich algebraic objects providing connections between various areas of mathematics such as representation theory \cite{MR2104671}, operator algebras \cite{1111.1362,MR1424954}, and quantum field theories in physics \cite{1312.7188,MR3146015}. Therefore, classification results regarding fusion categories have important implications for these various subjects. While a full classification result remains the distant goal, such a task is hopelessly out of reach with current available techniques. Instead research focuses on classifying ``small'' fusion categories, where small can have a variety of different meanings. For example one can attempt to classify fusion categories with a small number of simple objects \cite{MR3427429}.

Inspired by the successful classification of low index subfactors \cite{MR3166042,1509.00038} one can attempt a program to classify \change{unitary} fusion categories $\otimes$-generated by an object of small dimension. Given the success the low index subfactor classification in finding exotic examples \cite{0909.4099}, we hope that similar success can be obtained in finding exotic examples of fusion categories through this program. One of the earliest results in the field is the classification of \change{unitary} fusion categories generated by a self-dual object of dimension less than 2 \cite{MR1976459,MR1193933,MR1145672,MR1313457,MR1929335,MR1308617, MR1617550}, which contained the exceptional $E_6$ and $E_8$ examples. This result was generalised in \cite{1810.05717}, replacing the self-dual condition with a mild commutativity condition. Here again exceptional examples were found, namely the quantum subgroups $\mathcal{E}_4$ of $\mathfrak{sl}_4$ and $\mathcal{E}_{16,6}$ of $\mathfrak{sl}_2 \oplus\mathfrak{sl}_3$. Thus the program to classify \change{unitary} fusion categories $\otimes$-generated by an object of small dimension appears fruitful for discovering interesting new examples.

While the results \cite{MR1976459} and \cite{1810.05717} do provide a broad classification, it is somewhat unsatisfying having conditions on the $\otimes$-generating object of small dimension. However very little is known about the classification of categories $\otimes$-generated by an arbitrary object of small dimension. Outside of the trivial result for objects of dimension $1$, the easy proof of which has been known as folklore since the earliest days of the field (details can be found in \cite[Section 9]{1711.00645}), such classification results are non-existent in the literature. The purpose of this paper is to provide the first non-trivial such classification. The main result of this paper gives a complete classification of \change{unitary} fusion categories $\otimes$-generated by an object of dimension $\frac{1 + \sqrt{5}}{2}$.

\begin{theorem}\label{thm:main}
Let $\cC$ be a \change{unitary} fusion category $\otimes$-generated by an object of dimension $\frac{1 + \sqrt{5}}{2}$. Then $\cC$ is \changedd{unitarily} monoidally equivalent to either
\[ \operatorname{Fib}^{\boxtimes N} \overset{\omega}{\rtimes} \Z{NM} \text{ where } N,M\in \mathbb{N}\text{ and } \omega \in H^3(\Z{NM} , S^1) ,\]
where the action of $\Z{NM}$ factors through the action of $\Z{N}$ on $\operatorname{Fib}^{\boxtimes N}$ that cyclically permutes the $N$ factors,
or 
\[ \mathcal{TT}_{3}^{\boxtimes N} \overset{\omega}{\rtimes} \Z{2NM} \text{ where } N,M\in \mathbb{N}\text{ and } \omega \in H^3(\Z{2NM} , S^1),\]
where the $\Z{2NM}$ action factors through the action of $\Z{2N}$ on $\mathcal{TT}_{3}^{\boxtimes N}$ described in Remark~\ref{rmk:syms}.
\end{theorem}

To summarise the above theorem, we find all \change{unitary} fusion categories $\otimes$-generated by an object of dimension $\frac{1 + \sqrt{5}}{2}$ are constructed as wreath products of the category $\operatorname{Fib}$ or the category $\mathcal{TT}_{3}$, which is perhaps better known as the non-trivial ``fish-like'' quotient of $\operatorname{Fib}* \operatorname{Fib}$, or as the dual even part of the $2D2$ subfactor. While this classification does not reveal the existence of any truly exotic new fusion categories, it does find $\mathcal{TT}_{3} \rtimes \Z{2}$ as a \change{unitary} fusion category $\otimes$-generated by an object of dimension $\frac{1 + \sqrt{5}}{2}$. To the authors best knowledge, this example of a category $\otimes$-generated by an object of small dimension was not previously known.

Our arguments to prove Theorem~\ref{thm:main} reduce quickly to classifying finite \changedd{unitarizable} quotients of $\operatorname{Fib}^{*N}$ that have a symmetry between the $N$ distinct $\operatorname{Fib}$ generators. Hence as a by-product of our main theorem, we produce a classification of such quotients, extending results of Liu \cite{MR3345186} and Izumi-Morrison-Penneys \cite{MR3536926}. The lack of truly exotic categories in Theorem~\ref{thm:main} boils down to the surprising lack of interesting \changedd{unitarizable} quotients of $\operatorname{Fib}^{*N}$ for $N\geq 3$. Towards generalising Theorem~\ref{thm:main}, it appears that classifying cyclic quotients of $\cC^{*N}$ for $\cC$ will play a key role, as from such a quotient one can construct a new category $\otimes$-generated by an object of the same dimension as the $\otimes$-generator of $\cC$. In particular it would be interesting to classify finite quotients of free products of the Ising category, or of free products of the ``tadpole'' categories \cite{MR2559686}.

\begin{rmk}
\change{ We wish to point out that the assumption that the category $\cC$ is unitary in the above theorem can almost certainly be removed. The new theorem would classify fusion categories $\otimes$-generated by an object of Frobenius-Perron dimension $\frac{1 + \sqrt{5}}{2}$. The sticking point for proving this more general theorem, is that there only exists a classification of finite unitary quotients of $\operatorname{Fib} * \operatorname{Fib}$ in the current literature. If one could remove the unitary assumption on this classification result, and obtain a classification of all finite quotients of $\operatorname{Fib} * \operatorname{Fib}$, then the methods of this paper would directly generalise. The expanded classification theorem would now include the Galois conjugates of the categories in the above theorem, under the map $\sqrt{5} \mapsto -\sqrt{5}$.}

\changedd{
While the main Theorem of this paper is a statement of unitary fusion categories, for the majority of the proofs in this paper, we work in the unitarizable setting, that is fusion categories which have a unitary structure but where we have not chosen the unitary structure. This is for two reasons. We can't work in the purely algebraic setting because the results of Liu \cite{MR3345186} only apply to unitary categories. However, we also can't work in the purely unitary setting because the results of Galindo \cite[Theorem 4.1]{MR2796073} only apply in the algebraic setting. Adapting Liu's results to the purely algebraic setting might be quite difficult, while adapting Galindo's result to the unitary setting seems seems tractable.}
\end{rmk}

As every unitary fusion category embeds uniquely into bimodules of the hyperfinite type $\changed{\rm II_1}$ factor $\mathcal{R}$ \cite{MR1749868,MR1055708}, we thus have as a corollary to Theorem~\ref{thm:main} a (non-constructive) classification of finite depth bimodules of $\mathcal{R}$ with index $\frac{1 + \sqrt{5}}{2}$. The Jones index theorem implies that the three smallest dimensions of a finite depth bimodule of $\mathcal{R}$ are $1, \sqrt{2}$, and $\frac{1 + \sqrt{5}}{2}$. Hence the results of this paper, coupled with the known classification of unitary fusion categories generated by an object of dimension 1, give a complete classification of all finite depth bimodules of $\mathcal{R}$ for two of these three smallest dimensions. This motivates the classification of unitary fusion categories $\otimes$-generated by an object of dimension $\sqrt{2}$, which would give as a corollary a classification of finite depth bimodules of $\mathcal{R}$ in the range $[1, \frac{1 + \sqrt{5}}{2}]$.

Our paper is structured as follows.

In Section~\ref{sec:prelim} we give the necessary definitions for this paper. In particular we define semi-direct product categories, free product categories and quotients, and introduce the categories $\operatorname{Fib}$ and $\mathcal{TT}_3$.

In Section~\ref{sec:gold} we prove that every \changedd{unitarizable} fusion category $\otimes$-generated by an object of dimension $\frac{1 +\sqrt{5}}{2}$ is a semi-direct product of a certain quotient of $\operatorname{Fib}^{*N}$. The argument is fairly straightforward, consisting of fusion ring manipulation and Frobenius reciprocity. We end the section by discussing the feasibility of possible generalisations, including removing the \changedd{unitarizable} condition, and changing the dimension of the $\otimes$-generator.

Section~\ref{sec:quotient} represents the bulk of this paper. Here we classify finite \changedd{unitarizable} quotients of $\operatorname{Fib}^{*N}$ satisfying a certain symmetry condition. Using fusion ring arguments in a marathon case by case analysis, we show the surprising result that the only such quotients are $\operatorname{Fib}^{\boxtimes N}$ and $\mathcal{TT}_3^{\boxtimes \frac{N}{2} }$ if $N$ is even. Further, we show that any finite \changedd{unitarizable} quotient of $\operatorname{Fib}^{*N}$ factors through $\operatorname{Fib}^{\boxtimes n} \boxtimes \mathcal{TT}_3^{\boxtimes \frac{m}{2} }$ where $n + 2m = N$. Using this result it should be possible to classify all finite \changedd{unitarizable} quotients of $\operatorname{Fib}^{*N}$. We neglect to follow up on this generalised result as it goes beyond the scope of this paper.

We end the paper with Section~\ref{sec:class}, which ties together the results of this paper to prove Theorem~\ref{thm:main}.

\subsection*{Acknowledgements}
The author would like to thank Dietmar Bisch and Vaughan Jones for useful discussions regarding the problem tackled in this paper. We thank Corey Jones, Scott Morrison, and Dave Penneys for conversations that inspired the author to begin thinking about this problem. \change{We thank Noah Snyder for pointing out a crucial error in a previous version of this paper}. 

\section{Preliminaries}\label{sec:prelim}
For the basic theory of fusion categories we direct the reader to \cite{MR3242743}. Unless explicitly stated, all categories in this paper can be assumed to be fusion categories. That is, semisimple rigid tensor categories, with a finite number of simple objects, and simple unit.

Given an object $X$ in a fusion category $\cC$, then the category $\otimes$-generated by $X$ is the full subcategory of $\cC$ containing all summands of the objects $X^{\otimes n}$ for all $n\in \mathbb{N}$. The category $\otimes$-generated by $X$ is a fusion category by \cite[Lemma 3.7.6]{MR3242743}. We say $\cC$ is $\otimes$-generated by $X$ if the category generated by $X \in \cC$ is $\cC$. 
%
%
%

\subsection*{The Fibonacci category}

A recurring theme in this paper will be the appearance of the Fibonacci categories. A Fibonacci fusion category has two simple objects $\mathbf{1}$ and $\tau$, satisfying the fusion rule
\[  \tau \otimes \tau \cong \mathbf{1} \oplus \tau.\]
There exist two associators for this fusion rule, both of which can be found in \cite{MR2889539}. \change{Only one of these fusion categories has a unitary structure, which is unique.
\begin{dfn}
We write $\operatorname{Fib}$ for the unitary fusion category with Fibonacci fusion rules, and we write $\overline{\operatorname{Fib}}$ for the non-unitary fusion category with Fibonacci fusion rules.
\end{dfn}
}

The categories $\operatorname{Fib}$ and $\overline{\operatorname{Fib}}$ can be realised in a variety of different ways. For example $\operatorname{Fib}$ is monoidally equivalent to: the even part of the $A_4$ planar algebra, the adjoint category of the category of level $3$ integrable representations of $\widehat{\mathfrak{sl}}_2$, and the semi-simplification of the category of $U_{e^\frac{\pi i}{18}}(\mathfrak{g}_2)$ modules. Each of these categories is defined over $\mathbb{Q}[\sqrt{5}]$. The $\overline{\operatorname{Fib}}$ category can be realised as the Galois conjugate of $\operatorname{Fib}$ by $\sqrt{5} \mapsto -\sqrt{5}$.

The category $\operatorname{Fib}$ is the prototypical example of a \change{unitary} fusion category $\otimes$-generated by an object of dimension $\frac{1 +\sqrt{5}}{2}$. We will see later in this paper that every \change{unitary} fusion category $\otimes$-generated by an object of dimension $\frac{1 +\sqrt{5}}{2}$ can be constructed from $\operatorname{Fib}$, though in surprisingly complicated ways. As such, the category $\operatorname{Fib}$ will appear frequently in this paper. To help with computations we introduce the following lemma, allowing a sort of cancellation of Fibonacci objects.

\begin{lemma}\label{lem:towrite}
Let $\cC$ be a fusion category, and $\tau \in \cC$ an object satisfying the Fibonacci fusion rule $\tau \otimes \tau \cong \mathbf{1} \oplus \tau$. Then for any objects $X,Y \in \cC$ we have
\[ \tau\otimes X \cong \tau \otimes Y \implies X \cong Y.\]
\end{lemma}
\begin{proof}
Suppose 
\[\tau\otimes X \cong \tau \otimes Y,\]
then we can left multiply by $\tau$ to get
\[  X \oplus \tau\otimes X \cong Y \oplus \tau\otimes Y \cong  Y \oplus \tau\otimes X.\]

\changed{
For each simple object $W \in \cC$, we have
\[ \dim(W \to X \oplus \tau\otimes X) = \dim(W \to X ) + \dim(W \to \tau\otimes X)\]
and
\[ \dim(W \to Y \oplus \tau\otimes X) = \dim(W \to Y ) + \dim(W \to \tau\otimes X)\]

 As the objects $ X \oplus \tau\otimes X $ and $Y \oplus \tau\otimes X$ are isomorphic, we have
 \[ \dim(W \to X \oplus \tau\otimes X) = \dim(W \to Y \oplus \tau\otimes X).\]
 
Therefore, for each simple $W\in \cC$, we have
 \[   \dim(W \to X)  = \dim(W \to Y) . \]
 
 As $\cC$ is a semi-simple category, every object is uniquely determined, up to isomorphism, by the multiplicities of isomorphism classes of simple objects in its decomposition. Thus $X \cong Y$.
}


\end{proof}
\subsection*{Free products and quotients}

Given two fusion categories $\cC$ and $\mathcal{D}$, one can form the free product category $\cC * \mathcal{D}$. In $\cC * \mathcal{D}$ the objects consist of words of the objects in $\cC$ and $\mathcal{D}$, and the morphisms consist of non-crossing morphisms of both $\cC$ and $\mathcal{D}$. Additional details can be found in \cite{MR3536926,MR1437496}.

For this paper we will be interested in the free product category $\operatorname{Fib}^{*N}$ for $N\in \mathbb{N}_{\geq 1}$. If we write $\{\tau_i : 1 \leq i \leq N\}$ for the $N$ generators of the component $\operatorname{Fib}$ categories, then an object of $\operatorname{Fib}^{*N}$ will be of the form
\[  \tau_a  \tau_b \tau_c \cdots \qquad \text{where } a,b,c,\cdots \in \{1,2,\cdots N\}.\]
It is obvious that the category $\operatorname{Fib}^{*N}$ has infinitely many simple objects if $N\geq 2$, and hence is not fusion. However it is possible to take quotients of $\operatorname{Fib}^{*N}$ to obtain \changedd{unitarizable} fusion categories.

\begin{dfn}
A \changedd{unitarizable} quotient of $\operatorname{Fib}^{*N}$ is a rigid \changedd{unitarizable} tensor category $\cC$, along with a faithful dominant functor
\[ \operatorname{Fib}^{*N} \to \cC.     \]
\end{dfn}
Following \cite{MR3536926} it is helpful to work with the subcategory of a \changedd{unitarizable} quotient of $\operatorname{Fib}^{*N}$ which lies in the image of the faithful dominant functor. Taking the idempotent completion of this image category recovers the \changedd{unitarizable} quotient, hence there is no cost to working with this image category. With this perspective, we can think of a \changedd{unitarizable} quotient of $\operatorname{Fib}^{*N}$ as having the same objects as $\operatorname{Fib}^{*N}$, i.e words in $\{\tau_i : 1 \leq i \leq N\}$, but with more morphisms.  These additional morphisms give relations between the objects of the quotient, causing non-isomorphic objects in $\operatorname{Fib}^{*N}$ to become isomorphic, and simple objects to become non-simple. In particular we will be interested in \textit{finite \changedd{unitarizable} quotients of $\operatorname{Fib}^{*N}$}.

\begin{dfn}
A finite \changedd{unitarizable} quotient of $\operatorname{Fib}^{*N}$ is a \changedd{unitarizable} quotient that has finitely many simple objects.
\end{dfn}

A simple example of a \changedd{unitarizable} quotient of $\operatorname{Fib}^{*N}$ is the $N$-fold Deligne product $\operatorname{Fib}^{ \boxtimes N}$, where we impose the relation that the generators $\{\tau_i : 1 \leq i \leq N\}$ pairwise commute. The simple objects of $\operatorname{Fib}^{ \boxtimes N}$ are all words consisting of unordered non-repeating combinations of letters in $\{\tau_i : 1 \leq i \leq N\}$. Thus the rank of $\operatorname{Fib}^{ \boxtimes N}$ is equal to $2^N$, so $\operatorname{Fib}^{ \boxtimes N}$ gives our first example of a finite \changedd{unitarizable} quotient of $\operatorname{Fib}^{*N}$.

A more complicated example of a finite \changedd{unitarizable} quotient is the category $\mathcal{TT}_3$, constructed in \cite{MR3536926, MR3345186} and unpublished work of Izumi.
\begin{dfn}
The \change{unitary} fusion category $\mathcal{TT}_3$ has six simple objects $\mathbf{1}, f^{(2)}, \rho, \sigma, \overline{\sigma}$, and $\mu$, with dimensions $1, \left(\frac{1 + \sqrt{5}}{2}\right)^3, \frac{1 + \sqrt{5}}{2}, \left(\frac{1 + \sqrt{5}}{2}\right)^2,\left(\frac{1 + \sqrt{5}}{2}\right)^2,$ and $ \frac{1 + \sqrt{5}}{2}$ respectively. The fusion rules are given by
{\scriptsize{
$$
\begin{array}{|c|c|c|c|c|c|}
\hline 
\otimes & f^{(2)} & \rho & \sigma & \overline{\sigma} & \mu \\
\hline f^{(2)} & 1 {+} 2 f^{(2)} {+} \sigma {+} \overline{\sigma} {+} \rho {+} \mu & f^{(2)} {+} \overline{\sigma} & f^{(2)} {+} \sigma {+} \overline{\sigma} {+} \mu & f^{(2)} {+} \sigma {+} \overline{\sigma} {+} \rho & f^{(2)} {+} \sigma \\
\hline \rho & f^{(2)} {+} \sigma & 1 {+} \rho & f^{(2)} & \overline{\sigma} {+} \mu & \overline{\sigma} \\
\hline \sigma & f^{(2)} {+} \sigma {+} \overline{\sigma} {+} \rho & \sigma {+} \mu & f^{(2)} {+} \overline{\sigma} & 1 {+} f^{(2)} {+} \mu & f^{(2)} \\
\hline \overline{\sigma} & f^{(2)} {+} \sigma {+} \overline{\sigma} {+} \mu & f^{(2)} & 1 {+} f^{(2)} {+} \rho & f^{(2)} {+} \sigma & \overline{\sigma} {+} \rho \\
\hline \mu & f^{(2)} {+} \overline{\sigma} & \sigma & \sigma {+} \rho & f^{(2)} & 1 {+} \mu \\
\hline
\end{array}
$$
}}
\end{dfn}

The objects $\rho$ and $\mu$ in $\mathcal{TT}_3$ both satisfy Fibonacci fusion, and together they $\otimes$-generate all of $\mathcal{TT}_3$. Thus $\mathcal{TT}_3$ is a finite quotient of $\operatorname{Fib}^{*2}$. 

A theorem of Liu \cite{MR3345186} shows that the only finite \changedd{unitarizable} quotients of $\operatorname{Fib}^{*2}$ are $\operatorname{Fib}$, $\operatorname{Fib}^{\boxtimes 2}$, and $\mathcal{TT}_3$. Outside of this result little is known of finite \changedd{unitarizable} quotients of $\operatorname{Fib}^{*N}$. In Section~\ref{sec:quotient} of this paper we will extend Liu's result to classify finite \textit{rank preserving} \textit{cyclic} \changedd{unitarizable} quotients of $\operatorname{Fib}^{*N}$ for all $N$. 

\begin{dfn}
We say a \changedd{unitarizable} quotient of $\operatorname{Fib}^{*N}$ is rank preserving if $\tau_i \cong \tau_j$ implies $i = j$.
\end{dfn}

\begin{dfn}
We say a \changedd{unitarizable} quotient of $\operatorname{Fib}^{*N}$ is cyclic if there is a monoidal auto-equivalence mapping the generators
\[ \tau_1 \mapsto \tau_2 \mapsto \cdots \mapsto \tau_N \mapsto \tau_1.\]
\end{dfn}

\begin{example}
Consider the finite quotient $\mathcal{TT}_3$ of $\operatorname{Fib}^{*2}$. As $\rho \ncong \mu$ we have that $\mathcal{TT}_3$ is a rank preserving quotient. Further, one has from \cite[Lemma 3.4]{MR3394622} that the planar algebra for $\mathcal{TT}_3$ is generated by a single $6$-box $T$. From the relations for $T$ given in the same paper, we can see that the planar algebra for $\mathcal{TT}_3$ has a \change{*-}automorphism \change{mapping} $T \mapsto -T$. Thus by \cite[Theorem A]{1607.06041} we get a monoidal auto-equivalence of $\mathcal{TT}_3$ that sends
\[  \rho \leftrightarrow \mu , \qquad \text{ and } \quad \sigma \leftrightarrow \overline{\sigma}.\]
Thus $\mathcal{TT}_3$ is a finite rank preserving cyclic \changedd{unitarizable} quotient of $\operatorname{Fib}^{*2}$.
\end{example}

Our motivation for classifying such quotients will become apparent in Section~\ref{sec:gold}. We remark that the techniques of Section~\ref{sec:quotient} should extend in a fairly straightforward manner to classify all finite \changedd{unitarizable} quotients of $\operatorname{Fib}^{*N}$, not just finite rank preserving cyclic quotients. Infinite quotients of $\operatorname{Fib}^{*N}$ however still remain mysterious.
\subsection*{$G$-graded categories}
Let $G$ be a finite group, and $\cC$ a fusion category. We say $\cC$ is $G$-graded if we have an abelian decomposition 
\[  \cC \simeq \bigoplus_G \cC_g,\]
with each $\cC_g$ a non-trivial abelian subcategory of $\cC$, such that the tensor product of $\cC$ restricted to $\cC_g \times \cC_h$ has image in $\cC_{gh}$.

Consider the following monoidal subcategory of $\cC$, which we call the adjoint subcategory of $\cC$.
\[  \ad(\cC) := \langle X \otimes \overline{X} : X\in \operatorname{Irr}(\cC)\rangle.\]
The category $\cC$ is always graded by some group $G$, such that $\cC_e = \ad(\cC)$.

\subsection*{Semi-direct product categories}
An important construction for this paper will be the semi-direct product of a \change{unitary} fusion category $\cC$ with a finite group. The key ingredient for this construction is a \textit{categorical action} of $\cC$.

\begin{dfn}
Let $G$ be a finite group, and $\cC$ a \change{unitary} fusion category. A categorical action of $G$ on $\cC$ is a monoidal functor
\[   \rho : \underline{G} \to \underline{\operatorname{Eq}}(\cC),\]
where $\underline{G}$ is the monoidal category whose objects are the objects of $G$, and whose morphisms are identities, and $\underline{\operatorname{Eq}}(\cC)$ is the monoidal category of monoidal \change{*-}auto-equivalences of $\cC$.
\end{dfn}

Given a \change{unitary} fusion category with a categorical action we can construct the semi-direct product, a new \change{unitary} fusion category.

\begin{dfn}\label{def:semi}
Let $\cC$ be a \change{unitary} fusion category with a categorical $G$ action $\rho$, and let $\omega \in H^3(G,\change{S^1})$. We define the semi-direct product category $\cC \overset{\omega}{\rtimes} G$ as the abelian category
\[  \cC \overset{\omega}{\rtimes} G:=  \bigoplus_G \cC,\] 
with tensor product
\[   (X_1, g_1) \otimes (X_2,g_2) :=  (X_1 \otimes \rho(g_1)[X_2], g_1g_2), \]
and associator
\[   [ (X_1,g_1) \otimes (X_2,g_2)] \otimes (X_3,g_3) \to  (X_1,g_1) \otimes [(X_2,g_2) \otimes (X_3,g_3)] \]
given by the isomorphism
\[ \omega_{g_1,g_2,g_3} (\id_{X_1}\otimes \tau^{g_1}_{X_2,\rho(g_2)[X_3]} )\circ (\id_{X_1} \otimes \id_{\rho(g_1)[X_2]} \otimes \mu_{g1,g2}).      \]
Here $\tau^{g_1}$ is the tensorator for the monoidal functor $\rho(g_1)$, and $\mu$ is the tensorator for the monoidal functor $\rho$.
\change{The $*$-structure for $\cC \overset{\omega}{\rtimes} G$ is directly inherited from the $*$-structure for $\cC$.}
\end{dfn}

It is clear from the definition that $\cC \overset{\omega}{\rtimes} G$ is a $G$-graded category, and further, each graded piece contains an invertible element. The following theorem, due to Galindo, gives a converse to this observation, showing that every graded category with an invertible element in each graded piece must be a semi-direct product.

\begin{theorem}\label{thm:galindo}\cite[Theorem 4.1]{MR2796073}
Let $\cC$ be a $G$-graded fusion category such that each graded piece $\cC_g$ contains an invertible element $U_g$. Then $\cC$ is monoidally equivalent to
\[  \cC_e \overset{\omega}{\rtimes} G\]
for some $\omega \in H^3(G,\mathbb{C}^\times)$, where the categorical action $\rho$ of $G$ on $\cC_e$ is defined on objects by
\[   \rho_g(X) := U_g \otimes X \otimes U_{g^{-1}}.\]
\end{theorem}

\section{Initial reduction of classification}\label{sec:gold}
The main aim of this paper is to provide a complete classification of \change{unitary} fusion categories $\otimes$-generated by an object of dimension $\frac{1 +\sqrt{5}}{2}$. While initially this may seem like a hopelessly difficult task, we shall see in this Section that such categories all share the same uniform structure. That is they are all semi-direct products of certain finite \changedd{unitarizable} quotients of free products of arbitrarily many copies of $\operatorname{Fib}$. As semi-direct product categories are well understood, we have reduced our classification problem to the somewhat simpler task of classifying finite \changedd{unitarizable} quotients of free products. We will deal with such quotients in the next Section of this paper. For now we focus on proving the following theorem. The proof is fairly straightforward, and boils down to basic fusion ring arguments, and a lot of Frobenius reciprocity.

\begin{theorem}\label{thm:semi}
Let $\cC$ be a \changedd{unitarizable} fusion category $\otimes$-generated by an object of dimension $\frac{1 +\sqrt{5}}{2}$\changed{.} Then $\cC$ is monoidally equivalent to a semi-direct product of a finite rank preserving cyclic \changedd{unitarizable} quotient of $\operatorname{Fib}^{*N}$ by a cyclic group, with the generator of the cyclic group acting on the \change{quotient} by factoring through the given cyclic action.
\end{theorem}
\begin{proof}
Let $\cC$ be a \changedd{unitarizable} fusion category $\otimes$-generated by an object $X$ of dimension $\frac{1 +\sqrt{5}}{2}$. As $X\otimes  \overline{X}$ must contain the tensor unit, we have that 
\[  X \otimes \overline{X} \cong \mathbf{1} \oplus \tau\]
with $\tau$ some object in $\cC$. A dimension count shows $\tau$ has dimension $\frac{1 + \sqrt{5}}{2}$. In particular $\tau$ is a simple object. Further, as both $\mathbf{1}$ and $X \otimes \overline{X}$ are self-dual, we also have that $\tau$ is self-dual. Hence $\tau$ generates a $\operatorname{Fib}$ subcategory of $\cC$, and in particular satisfies the fusion rule $\tau \otimes \tau \cong \mathbf{1} \oplus \tau$.

As $\dim(X \otimes \overline{X} \to \tau) = 1$, we can apply Frobenius reciprocity to get $\dim(X \to \tau \otimes X) = 1$. As $X$ is simple, we get $X \subseteq \tau \otimes X$ and so 
\[
   \tau \otimes X \cong X \oplus g
\]
with $g$ an object of $\cC$. Another dimension count shows $g$ is a simple object of dimension $1$. In particular $\dim(\tau \otimes X \to g) = 1$, and another application of Frobenius reciprocity gives $\dim( \tau\otimes g \to X) = 1$. As both $\tau\otimes g$ and $X$ are simple, we thus have 
\begin{equation}\label{eq:Xtau}
  X \cong \tau \otimes g.
 \end{equation}

As $\cC$ is a fusion category, the order of $g$ is finite. Let $N$ be the smallest integer (which must exist as the order of $g$ is finite) such that $g^N \otimes \tau \otimes \overline{g}^N = \tau$. Consider the $N$ distinct simple objects
\[     \tau_n := g^n \otimes \tau \otimes \overline{g}^n : n = 0,1,\cdots , N-1 .\]
It is straightforward to verify that each of these distinct objects satisfies the fusion 
\[  \tau_n \otimes \tau_n \cong \mathbf{1} \oplus \tau_n,\]
and further each $\tau_n$ has dimension $\frac{1 +\sqrt{5}}{2}$. Hence each $\tau_n$ generates a $\operatorname{Fib}$ subcategory of $\cC$, and so the collection of objects $\{\tau_n : n = 1,2,\cdots , N\}$ together $\otimes$-generate a subcategory of $\cC$ that is equivalent to a finite rank preserving quotient of $\operatorname{Fib}^{*N}$. As $\cC$ is \changedd{unitarizable}, this finite rank preserving quotient must also be \changedd{unitarizable}.

As $X$ $\otimes$-generates $\cC$, the adjoint subcategory of $\cC$ is $\otimes$-generated by the objects
\[   X^n\otimes \overline{X}^n \text{ for } n\in \mathbb{N}.\]
Using Equation~\eqref{eq:Xtau} we can write
\[   X^n\otimes \overline{X}^n \cong (\tau \otimes g)^n \otimes  (\overline{g} \otimes \tau)^n \cong \left(\bigotimes_{i= 0}^{n-1}  \tau_i \right) \otimes \left(\bigotimes_{i= n-1}^{0}  \tau_i \right)  .\]
Hence, the adjoint subcategory of $\cC$ is $\otimes$-generated by the $N$ simple objects $\{\tau_n\}$, and thus the adjoint subcategory of $\cC$ is equivalent to a finite rank preserving quotient of $\operatorname{Fib}^{*N}$.

As $\cC$ is $\otimes$-generated by $X$, the category $\cC$ is a cyclic extension of $\ad(\cC)$, with $X$ living in the 1-graded component of this cyclic grading (written additively). Furthermore, Equation~\eqref{eq:Xtau} shows that $g$ also lives in the 1-graded component of the cyclic grading, which implies that every graded component of the grading has an invertible object. Thus Theorem~\ref{thm:galindo} implies that $\cC$ is a semi-direct product of its adjoint subcategory by a cyclic group, and hence is a semi-direct product of a finite rank preserving \changedd{unitarizable} quotient of $\operatorname{Fib}^{*N}$ by a cyclic group. The same theorem shows the generator of the cyclic group acts on the generators of $\operatorname{Fib}^{*N}$ by sending $\tau_i \mapsto \tau_{i+1}$. Hence the finite rank preserving \changedd{unitarizable} quotient of $\operatorname{Fib}^{*N}$ must also be cyclic.
\end{proof}

Before we move on to classifying finite rank preserving cyclic quotients of $\operatorname{Fib}^{*N}$, let us discuss possible generalisations of the above theorem. The key step in the above proof is finding the invertible object $g$, which lives in the same graded component of $\cC$ as the generator $X$. With this object in hand we know we are looking at a semi-direct product category, and hence classification is possible. If we consider a generator of dimension $\sqrt{2}$ then it is fairly easy to convince oneself that we can never get such an invertible object, the problem being that we run into issues with the grading. However if we restrict to objects of dimension $2\cos\left(\frac{\pi}{2N+1}\right)$ then these grading issues seem to not appear. Hence it should be possible to prove a generalisation of the above theorem for fusion categories $\otimes$-generated by an object of dimension $2\cos\left(\frac{\pi}{2N+1}\right)$. We stress that we have not worked through the details of this generalisation, and do not claim it as conjecture.

Our lack of current interest in the mentioned generalisation stems from the fact we would have to classify finite \changedd{unitarizable} quotients of $T_N * T_N$ for all $N$ to provide a generalised classification result. Given the difficulty of the $T_2 *T_2$ case \cite{MR3345186,MR3536926}, the general $N$ case appears completely out of reach. Supposing a breakthrough insight is discovered, allowing for the classification of finite \changedd{unitarizable} quotients of $T_N * T_N$ for all $N$, then it is likely the techniques from this paper would generalise to provide a generalised classification result.

The other generalisation one can consider is to relax the \changedd{unitarizable} condition on the category $\cC$, and to instead consider fusion categories $\otimes$-generated by an object of Frobenius-Perron dimesion $\frac{1 +\sqrt{5}}{2}$. The above theorem has a straightforward generalisation, now showing that any such category is a semi-direct product of either $\operatorname{Fib}^{*N}$ or of the Galois conjugate $\overline{\operatorname{Fib}}^{*N}$. The issue with pursuing this generalisation is that \change{there is currently no classification of non-\changedd{unitarizable} quotients of $\operatorname{Fib}*\operatorname{Fib}$}. Thus the results of the following Section can not be extended to give a classification of \change{not necessarily \changedd{unitarizable}} finite cyclic quotients of $\operatorname{Fib}^{*N}$. However, this is the only stumbling block. \change{Assuming the expected result that every quotient of $\operatorname{Fib}*\operatorname{Fib}$ is \changedd{unitarizable}}, then the rest of this paper would generalise directly to give a complete classification of fusion categories $\otimes$-generated by an object of Frobenius-Perron dimension $\frac{1 +\sqrt{5}}{2}$.

\section{Finite Cyclic Quotients of $\operatorname{Fib}^{*N}$}\label{sec:quotient}

With Theorem~\ref{thm:semi} in mind, we have motivation to classify finite rank preserving cyclic \changedd{unitarizable} quotients of $\operatorname{Fib}^{*N}$. This is \textit{a priori} a very difficult task, given the complexity that was required to classify finite \changedd{unitarizable} quotients of $\operatorname{Fib} *\operatorname{Fib}$. However, miraculously we are able to prove the following classification result, purely using fusion ring arguments.

\begin{theorem}\label{thm:fin}
Let $\cC$ be a rank preserving finite cyclic \changedd{unitarizable} quotient of $\operatorname{Fib}^{*N}$, then $\cC$ is monoidally equivalent to either
\[  \operatorname{Fib}^{\boxtimes N},\]
or, if $N$ is even, 
\[   \mathcal{TT}_3^{\boxtimes \frac{N}{2}}.\]
\end{theorem}

Essentially this theorem shows that the only interesting finite cyclic \changedd{unitarizable} quotients of $\operatorname{Fib}^{*N}$ are rank 2. While we did find this result surprising (and somewhat disappointing, as it excludes the possibility of potential new exotic \changedd{unitarizable} fusion categories), it does have some parallels to \changed{a classical group theory result. 

One can consider the category $\operatorname{Fib}$ as being the higher categorical analogue of the group $\Z{2}$, in the sense that $\operatorname{Fib}$ has two simple objects, but with non-pointed fusion. With this analogy in mind, we can draw a further analogy between the classification of finite quotients of $\operatorname{Fib}^{*N}$, and the Coxeter classification of finite quotients of $\Z{2}^{*N}$ \cite{coxeter}. 

In the case where $N=2$, the classical Coxeter classification simply gives the dihedral groups, while the quantum analogue follows from Liu's classification of quotients of $\operatorname{Fib}*\operatorname{Fib}$. While there are an infinite family of dihedral groups, Liu's classification shows that only the first two of the quantum counterparts exist. Hence, it seems reasonable to expect the that the classification of finite quotients of $\operatorname{Fib}^{*N}$ should be more significantly more restrictive than the classification of finite quotients of $\Z{2}^{*N}$. Given that the examples appearing in Coxeter's classification were sparse to begin with, this point of view helps to explain the surprising lack of interesting examples in Theorem~\ref{thm:fin}. 

}

 Let us briefly sketch the proof of Theorem~\ref{thm:fin}.

We begin by considering the finite \changedd{unitarizable} quotients of $\operatorname{Fib}^{*3}$, with generators $a$, $b$, and $c$. We are able to show that we must have a relation of the form
\[ \underbrace{abcab\cdots }_\text{length n}   \cong  \underbrace{bcabc\cdots }_\text{length n}, \]
for some finite $n$. Further, as each pair of generators must generate a finite \changedd{unitarizable} quotient of $\operatorname{Fib}^{*2}$, we must have the further three relations 
\[   ab \cong ba \qquad \text{ or } \qquad aba \cong bab \]
and 
\[   bc \cong cb \qquad \text{ or } \qquad bcb \cong cbc \]
and 
\[   ac \cong ca \qquad \text{ or } \qquad aca  \cong cac. \]

A marathon case by case analysis shows that in fact, one of the generators commutes with the other two. With this result in hand one can quickly prove that any rank preserving finite \changedd{unitarizable} quotient of $\operatorname{Fib}^{*N}$ must factor through 
\[  \operatorname{Fib}^{\boxtimes n} \boxtimes \mathcal{TT}_3^{\boxtimes m}\]
where $n + \frac{m}{2} =N$.

At this point we use our cyclic quotient assumption to see that a rank preserving cyclic \changedd{unitarizable} quotient of $\operatorname{Fib}^{*N}$ must be a quotient of either $\operatorname{Fib}^{\boxtimes N}$, or $ \mathcal{TT}_3^{\boxtimes \frac{N}{2}}$ if $N$ is even. Finally we show that only the trivial quotients of these categories exist.

The results of this Section, while not particularly technical, are extremely long and messy. A ``better'' proof could involve generalising the results of \cite{MR3345186} to give a finite bound on the $n$ such that 
\[ \underbrace{abcab\cdots }_\text{length n}   \cong  \underbrace{bcabc\cdots }_\text{length n}. \]
From here, one would now just have a finite number of cases to consider to show that one of the generators commutes with the other two, significantly shortening the length of the proof.

For now we just have our marathon brute force proof. Before we begin, we introduce some notation
\begin{dfn}
Let $\cC$ be a \changedd{unitarizable} quotient of $\operatorname{Fib}^{*3}$. We say a word in $\cC$ is cyclic if the letters are tri-alternating.
\end{dfn}

\begin{example}
Consider a \changedd{unitarizable} quotient of $\operatorname{Fib}^{*3} = \langle \tau_1, \tau_2, \tau_3 \rangle $. The words 
\[   \tau_1 \tau_2 \tau_3 \tau_1 \tau_2  \quad \text{ and } \quad \tau_3 \tau_2 \tau_1 \tau_3\]
are cyclic, but the word
\[  \tau_1 \tau_2 \tau_3 \tau_2 \]
is not.
\end{example}

We begin with the proofs of this Section with the following lemma, inspired by \cite[Proposition 3.2]{MR3536926}.

\begin{lemma}\label{lem:simps}
Let $\cC$ be a \changedd{unitarizable} quotient of $\operatorname{Fib}^{*3}$ with generators $a$, $b$, and $c$, and let $n\in \mathbb{N}$ be the smallest such that either
\[ \underbrace{abcabc\cdots }_\text{length n}   \cong  \underbrace{bcabca\cdots }_\text{length n}\]
or
\[ \underbrace{bcabca\cdots }_\text{length n}   \cong  \underbrace{cabcab\cdots }_\text{length n}\]
or
\[ \underbrace{cabcab\cdots }_\text{length n}   \cong  \underbrace{abcabc\cdots }_\text{length n}.\]
then all cyclic words of length $n$ or less are simple. If there is no such $n$, then all cyclic words are simple.
\end{lemma}
\begin{proof}
This proof inducts on the length $k$ of the cyclic word.

\begin{trivlist}\leftskip=2em
\item \textbf{$k=0$:}

As $\cC$ is a tensor category, the tensor unit is simple.

\vspace{1em}

\item \textbf{$k=1$:}

As the objects $a$, $b$, and $c$ have dimension $\frac{1 +\sqrt{5}}{2}$, they must be simple.

\vspace{1em}

\item Inductive step for $k\leq n$:

Suppose that all cyclic words of length $k-1$ or less are simple.  Consider a cyclic word of length $k$, we consider nine cases, depending on if the word starts with $a$, $b$, or $c$, and if the word ends in $a$, $b$, or $c$. We work though the case where the cyclic word begins with $a$ and ends with $b$, and leave the other eight cases to the reader, as they are near identical. We compute

\begin{align*}
\dim(  \underbrace{abc\cdots cab }_\text{length k} \to \underbrace{abc\cdots cab }_\text{length k}) =& \dim(  \underbrace{aabc\cdots bca }_\text{length k } \to \underbrace{bca\cdots cabb }_\text{length k }) \\  
															=& \dim(  \underbrace{abc\cdots bca }_\text{length k-1 }\oplus  \underbrace{bca\cdots bca }_\text{length k-2 } \to\underbrace{bca\cdots bca }_\text{length k-2 }\oplus  \underbrace{bca\cdots cab}_\text{length k-2 )}) \\
															=& \dim(  \underbrace{abc\cdots bca }_\text{length k-1 } \to\underbrace{bca\cdots bca }_\text{length k-2 }) +\dim(  \underbrace{bca\cdots bca }_\text{length k-2 } \to\underbrace{bca\cdots bca }_\text{length k-2 }) \\ &+ \dim(  \underbrace{abc\cdots bca }_\text{length k-1 } \to\underbrace{bca\cdots cab }_\text{length k-1 }) + \dim(  \underbrace{bca\cdots bca }_\text{length k-2 } \to\underbrace{bca\cdots cab }_\text{length k-1 })\\
															=&  0+ 1+0+0 = 1.
\end{align*}
Where the last step uses the inductive hypothesis, along with the assumption in the statement of the lemma that 
\[\underbrace{abc\cdots bca }_\text{length k-1 } \ncong \underbrace{bca\cdots cab }_\text{length k-1 }.\]
Thus $ \underbrace{abc\cdots cab }_\text{length k}$ is simple.
\end{trivlist}
We can repeat the inductive step until $k = n+1$, hence proving the statement of the lemma.
\end{proof}

Restricting our attention to finite \changedd{unitarizable} quotients, we get the following Corollary.

\begin{cor}\label{cor:minn}
Let $\cC$ be a finite \changedd{unitarizable} quotient of $\operatorname{Fib}^{*3}$ with generators $a$, $b$, and $c$. Then, \changed{up to a relabelling of the three generators},  there exists an $n\in \mathbb{N}$ such that 
\[ \underbrace{abcab\cdots }_\text{length n}   \cong  \underbrace{bcabc\cdots }_\text{length n}\]
and that all cyclic words of length $n$ or less are simple.
\end{cor}
\begin{proof}
If there didn't exist a $n\in \mathbb{N}$ such that either
\[ \underbrace{abcab\cdots }_\text{length n}   \cong  \underbrace{bcabc\cdots }_\text{length n}\]
or
\[ \underbrace{bcabc\cdots }_\text{length n}   \cong  \underbrace{cabca\cdots }_\text{length n}\]
or
\[ \underbrace{cabca\cdots }_\text{length n}   \cong  \underbrace{abcab\cdots }_\text{length n},\]
then Lemma~\ref{lem:simps} would imply that $\cC$ would have an infinite number of simple objects, contradicting the assumption that $\cC$ is a finite \changedd{unitarizable} quotient. Thus there must exist such a $n\in \mathbb{N}$. Assume that this $n$ is chosen to be minimal, and further, use up a degree of freedom between the labels of the generators $a$, $b$, and $c$ so that $n$ is the smallest integer such that 
\[ \underbrace{abcab\cdots }_\text{length n}   \cong  \underbrace{bcabc\cdots }_\text{length n}.\]
As $n$ was chosen minimally, we have that all cyclic words of length $n-1$ or less are distinct. Thus Lemma~\ref{lem:simps} gives that all cyclic words of length $n$ or less are simple.
\end{proof}

Hence we have shown that any finite \changedd{unitarizable} quotient of $\operatorname{Fib}^{*3}$ must have a relation of the form
\[ \underbrace{abcab\cdots }_\text{length n}   \cong  \underbrace{bcabc\cdots }_\text{length n}\]
for some $n$. We will now show that this relation is inconsistent with most other relations we can put on $\operatorname{Fib}^{*3}$. Key to finding these inconsistencies will be the following non-isomorphisms.

\begin{lemma}\label{lem:shift}
Let $\cC$ be a rank preserving \changedd{unitarizable} quotient of $\operatorname{Fib}^{*3}$ with generators $a$, $b$, and $c$ satisfying the relations 
\[  aba \cong bab, \qquad bcb \cong cbc,\quad \text{ and } \quad aca \cong cac.\]
Then we have for any $n\in \mathbb{N}$
\begin{align*}
\underbrace{abcabc\cdots }_\text{length $n$}\quad  \ncong\quad &   b \underbrace{abcabc\cdots }_\text{length $n-1$} \\
\underbrace{bcabca\cdots }_\text{length $n$}\quad \ncong \quad&   c \underbrace{bcabca\cdots }_\text{length $n-1$} \\
\underbrace{cabcab\cdots }_\text{length $n$} \quad\ncong\quad &   a \underbrace{cabcab\cdots }_\text{length $n-1$} \\
\underbrace{acbacb\cdots }_\text{length $n$} \quad\ncong \quad&   c  \underbrace{acbacb\cdots }_\text{length $n-1$} \\
\underbrace{bacbac\cdots }_\text{length $n$} \quad\ncong\quad &  a  \underbrace{bacbac\cdots }_\text{length $n-1$} \\
\underbrace{cbacba\cdots }_\text{length $n$} \quad\ncong \quad&  b  \underbrace{cbacba\cdots }_\text{length $n-1$}.
\end{align*}
\end{lemma}
\begin{proof}
We induct on the length $n$.
\begin{trivlist}\leftskip=2em
\item \textbf{$n = 1$:}

This case follows as $\cC$ is a rank preserving quotient, so $a\ncong b \ncong c$.

\vspace{1em}

\item \textbf{$n = 2$:}

If $ab \cong ba$ then we can use the relation $aba =bab$ to get 
\[  b \oplus ab \cong a \change{\oplus} ab\]
which implies $a \cong b$, contradicting the assumption that $\cC$ is a rank preserving quotient. Thus $ab \ncong ba$. 

A similar argument shows $bc \ncong cb$ and $ac \ncong ca$.

\vspace{1em}

\item Inductive Step:

Suppose the result holds true for $n-2$, and aiming for a contradiction, assume 
\[\underbrace{abcabc\cdots }_\text{length $n$} \cong    b \underbrace{abcabc\cdots }_\text{length $n-1$}.\]
Using the relation $aba \cong bab$ we can rewrite the right hand side to get
\[\underbrace{abcabc\cdots }_\text{length $n$} \cong   aba \underbrace{cabcab\cdots }_\text{length $n-3$}.\]
We now apply Lemma~\ref{lem:towrite} twice to get
\[\underbrace{cabcab\cdots }_\text{length $n-2$} \cong   a \underbrace{cabcab\cdots }_\text{length $n-3$},\]
which is a contradiction to the inductive hypothesis. Hence 
\[\underbrace{abcabc\cdots }_\text{length $n$} \ncong   b \underbrace{abcabc\cdots }_\text{length $n-1$}.\]
The other $5$ cases follow near identically.
\end{trivlist}
\end{proof}

We now prove the main technical lemma of this Section, showing that any finite quotient of $\operatorname{Fib}^{*3}$ must have a large amount of commutativity.

\begin{lemma}\label{lem:hard}
Let $\cC$ be a rank preserving finite \changedd{unitarizable}quotient of $\operatorname{Fib}^{*3}$ with generators $a$, $b$, and $c$. Then one of the generators commutes with the other two.
\end{lemma}

 As mentioned earlier the proof of this lemma consists of brute force case bashing. We begin by setting up the cases. \changed{Let us write $a$, $b$, and $c$ for the three generators of this quotient. Using Corollary~\ref{cor:minn} we can arrange the generators $a,b$, and $c$ in such a way so that there exists a $n \in \mathbb{N}$ such that
\begin{equation}\label{eq:perm} 
\underbrace{abcab\cdots }_\text{length n}   \cong  \underbrace{bcabc\cdots }_\text{length n}
\end{equation}
and such that all cyclic words of length $n$ or less are simple.

With this fixed labelling of the generators, we can consider the four distinct cases:

\begin{itemize}
\item None of $a$, $b$, or $c$ commute with each other,
\item Only $a$ and $b$ commute with each other, 
\item Only $b$ and $c$ commute with each other, 
\item Only $a$ and $c$ commute with each other.
\end{itemize}

We are able to show that each of these four cases can not occur, hence showing that one of the generators commutes with the other two. To help with readability, we will deal with each of these cases in separate lemmas.

\begin{sublOne}
There are no rank preserving finite \changedd{unitarizable} quotients of $\operatorname{Fib}^{*3}$ such that none of $a$, $b$, or $c$ commute with each other.
\end{sublOne}}

\begin{proof}
Consider such a quotient. As none of $a$, $b$, or $c$ commute with each other, we must have from the classification of finite \change{unitary} quotients of $\operatorname{Fib}*\operatorname{Fib}$ the relations 
\begin{equation}
\label{eq:zippers} aba \cong bab,\qquad bcb \cong cbc,\quad  \text{ and }\quad  aca \cong cac.
\end{equation}
Here we split the proof up into three cases, depending on $n \mod 3$.
\begin{trivlist}\leftskip=3em
\item \textbf{Case: $n \equiv 0 \pmod 3$}

If $n \equiv 0 \pmod 3$, then the cyclic word $ \underbrace{bcabc\cdots }_\text{length n}$ ends in $a$. Consider the word
\[  \underbrace{bcabc\cdots abc }_\text{length $n-1$}  a \underbrace{cba\cdots cbacb }_\text{length $n-1$}.\]

If $n \equiv 0 \pmod 6$ then we can use the relations~\eqref{eq:zippers} to write 
\[ \underbrace{bcabc\cdots abc }_\text{length $n-1$}  a \underbrace{cba\cdots cbacb }_\text{length $n-1$} \cong \underbrace{cbacb\cdots acb }_\text{length $n-1$}  a \underbrace{bca\cdots bcabc }_\text{length $n-1$} ,\]
and if $n \equiv 3 \pmod 6$ then we can use the relations~\eqref{eq:zippers} to write 
\[ \underbrace{bcabc\cdots abc }_\text{length $n-1$}  a \underbrace{cba\cdots cbacb }_\text{length $n-1$} \cong \underbrace{bacba\cdots cba}_\text{length $n-1$}  c \underbrace{abc\cdots abcab }_\text{length $n-1$}.\]
Hence we further break the proof up into two more sub-cases.

\begin{trivlist}\leftskip=4.5em
\item \textbf{Case: $n \equiv 0 \pmod 6$}

As $n \equiv 0 \pmod 6$ we have the relation
\[ \underbrace{bcabc\cdots abc }_\text{length $n-1$}  a \underbrace{cba\cdots cbacb }_\text{length $n-1$} \cong \underbrace{cbacb\cdots acb }_\text{length $n-1$}  a \underbrace{bca\cdots bcabc }_\text{length $n-1$}.\]
We can use relation~\eqref{eq:perm} on the first $n$ letters of the left hand side of the above equation to get
\begin{equation}\label{eq:nfold}
 \underbrace{abcab\cdots cab }_\text{length $n-1$}  c \underbrace{cba\cdots cbacb }_\text{length $n-1$}\cong \underbrace{cbacb\cdots acb }_\text{length $n-1$}  a \underbrace{bca\cdots bcabc }_\text{length $n-1$}.
\end{equation}
Evaluating the left hand side using the $\operatorname{Fib}$ fusion rules reveals that $a$ is a simple summand. Hence
\[  \dim (a \to \underbrace{cbacb\cdots acb }_\text{length $n-1$}  a \underbrace{bca\cdots bcabc }_\text{length $n-1$}) > 0,\]
and an application of Frobenius reciprocity gives
\[  \dim ( a \underbrace{cbacb\cdots acb }_\text{length $n-1$}\to   \underbrace{cba\cdots cbacb }_\text{length $n-1$}a    ) > 0.\]
As cyclic words of length $n$ are simple, we thus have
\[   \underbrace{cba\cdots cba }_\text{length $n$} \cong  \underbrace{acb\cdots acb }_\text{length $n$}.\]
Returning to equation~\eqref{eq:nfold} we can now substitute in the above relation to the first $n$ letters of the right hand side to get 
\[ \underbrace{abcab\cdots cab }_\text{length $n-1$}  c \underbrace{cba\cdots cbacb }_\text{length $n-1$} \cong  a \underbrace{cbacb\cdots acb }_\text{length $n-1$} \underbrace{bca\cdots bcabc }_\text{length $n-1$},\]
to which we can apply Lemma~\ref{lem:towrite} to get
\[ \underbrace{bcab\cdots cab }_\text{length $n-2$}  c \underbrace{cba\cdots cbacb }_\text{length $n-1$} \cong  \underbrace{cbacb\cdots acb }_\text{length $n-1$} \underbrace{bca\cdots bcabc }_\text{length $n-1$}.\]
Expanding the left hand of this equation reveals a simple $b$ summand, and expanding the right hand side completely gives
\[  b \subset \left(\bigoplus_{i = 1}^{n-1}  \underbrace{cbacba \cdots }_\text{length $i$} \otimes\underbrace{\cdots abcabc }_\text{length $i-1$}\right) \oplus \mathbf{1}.\]
Thus there exists an $i$ such that 
\[  \dim ( b \to \underbrace{cbacba \cdots }_\text{length $i$} \otimes\underbrace{\cdots abcabc }_\text{length $i-1$}) > 0.\]
Applying Frobenius reciprocity gives 
\[  \dim ( b \underbrace{ cbacba\cdots }_\text{length $i-1$} \to \underbrace{cba cba\cdots }_\text{length $i$})  > 0,\]
which, as $\underbrace{cbacba \cdots }_\text{length $i$}$ is simple, implies that 
\[   b \underbrace{ cbacba \cdots }_\text{length $i-1$} \cong \underbrace{cbacba \cdots }_\text{length $i$}.\]
However this is a contradiction to Lemma~\ref{lem:shift}.

\vspace{1em}

\item \textbf{Case: $n \equiv 3 \pmod 6$}

As $n \equiv 3 \pmod 6$ we have the relation
\[ \underbrace{bcabc\cdots abc }_\text{length $n-1$}  a \underbrace{cba\cdots cbacb }_\text{length $n-1$} \cong \underbrace{bacba\cdots cba}_\text{length $n-1$}  c \underbrace{abc\cdots abcab }_\text{length $n-1$}.\]
We can use relation~\eqref{eq:perm} on the first $n$ letters of the left hand side of the above equation to get
\[
 \underbrace{abcab\cdots cab }_\text{length $n-1$}  c \underbrace{cba\cdots cbacb }_\text{length $n-1$} \cong \underbrace{bacba\cdots cba}_\text{length $n-1$}  c \underbrace{abc\cdots abcab }_\text{length $n-1$}.
\]
Evaluating the left hand side using the $\operatorname{Fib}$ fusion rules reveals that $a$ is a simple summand. Hence
\[  \dim (a \to \underbrace{bacba\cdots cba}_\text{length $n-1$}  c \underbrace{abc\cdots abcab }_\text{length $n-1$}) > 0,\]
and an application of Frobenius reciprocity gives
\[  \dim (a  \underbrace{bacba\cdots cba}_\text{length $n-1$} \to  \underbrace{bacba\cdots cbac}_\text{length $n$} ) > 0.\]
As $\underbrace{bacba\cdots cbac}_\text{length $n$}$ is simple we thus have
\[ a  \underbrace{bacba\cdots cba}_\text{length $n-1$} \cong  \underbrace{bacba\cdots cbac}_\text{length $n$}.\]
However this is a contradiction to Lemma~\ref{lem:shift}.
\end{trivlist}

\item \textbf{Case: $n \equiv 1 \pmod 3$}

If $n \equiv 1 \pmod 3$, then the cyclic word $ \underbrace{bcabc\cdots }_\text{length n}$ ends in $b$. Consider the word
\[  \underbrace{bca\cdots bca }_\text{length $n-1$}  b \underbrace{acb\cdots acb }_\text{length $n-1$}.\]
If $n \equiv 1 \pmod 6$ then we can use the relations~\eqref{eq:zippers} to write 
\[  \underbrace{bca\cdots bca }_\text{length $n-1$}  b \underbrace{acb\cdots acb }_\text{length $n-1$} \cong \underbrace{bac\cdots bac}_\text{length $n-1$}  b \underbrace{cab\cdots cab }_\text{length $n-1$} ,\]
and if $n \equiv 4 \pmod 6$ then we can use the relations~\eqref{eq:zippers} to write 
\[  \underbrace{bca\cdots bca }_\text{length $n-1$}  b \underbrace{acb\cdots acb }_\text{length $n-1$} \cong  \underbrace{cba\cdots cba }_\text{length $n-1$}  c \underbrace{abc\cdots abc }_\text{length $n-1$}.\]
Hence we further break the proof up into two more sub-cases.

\begin{trivlist}\leftskip=4.5em
\item \textbf{Case: $n \equiv 1 \pmod 6$}

If $n \equiv 1 \pmod 6$ then we have
\[  \underbrace{bca\cdots bca }_\text{length $n-1$}  b \underbrace{acb\cdots acb }_\text{length $n-1$} \cong \underbrace{bac\cdots bac}_\text{length $n-1$}  b \underbrace{cab\cdots cab }_\text{length $n-1$} ,\]
to which we can apply relation~\eqref{eq:perm} on the first $n$ letters of the left hand side of the above equation to get
\[  \underbrace{abc\cdots abc }_\text{length $n-1$}  a \underbrace{acb\cdots acb }_\text{length $n-1$} \cong \underbrace{bac\cdots bac}_\text{length $n-1$}  b \underbrace{cab\cdots cab }_\text{length $n-1$} .\]
Evaluating the left hand side reveals that
\[ a \subset  \underbrace{bac\cdots bac}_\text{length $n-1$}  b \underbrace{cab\cdots cab }_\text{length $n-1$},\]
from which we can use a similar argument as in the $n \equiv 3 \pmod 6$ case to see that
\[      a \underbrace{bac\cdots bac }_\text{length $n-1$} \cong  \underbrace{bac\cdots acb}_\text{length $n$}.\] 
However this is a contradiction to Lemma~\ref{lem:shift}.

\vspace{1em}

\item \textbf{Case: $n \equiv 4 \pmod 6$}

If $n \equiv 4 \pmod 6$ then we have
\[  \underbrace{bca\cdots bca }_\text{length $n-1$}  b \underbrace{acb\cdots acb }_\text{length $n-1$} \cong  \underbrace{cba\cdots cba }_\text{length $n-1$}  c \underbrace{abc\cdots abc }_\text{length $n-1$},\]
to which we can apply relation~\eqref{eq:perm} on the first $n$ letters of the left hand side of the above equation to get
\begin{equation}\label{eq:2ndquant}
  \underbrace{abc\cdots abc }_\text{length $n-1$}  a \underbrace{acb\cdots acb }_\text{length $n-1$} \cong \underbrace{cba\cdots cba}_\text{length $n-1$}  c \underbrace{abc\cdots abc }_\text{length $n-1$} .
\end{equation}
Evaluating the left hand side reveals that
\[ a \subset  \underbrace{cba\cdots cba}_\text{length $n-1$}  c \underbrace{abc\cdots abc }_\text{length $n-1$},\]
from which we can use a similar argument as in the $n \equiv 0 \pmod 6$ case to see that
\[      \underbrace{acb\cdots cba }_\text{length $n$} \cong  \underbrace{cba\cdots bac}_\text{length $n$}.\] 
We can apply the above relation to the first $n$ letters on the right hand side of Equation~\eqref{eq:2ndquant} to get
\[    \underbrace{abc\cdots abc }_\text{length $n-1$}  a \underbrace{acb\cdots acb }_\text{length $n-1$} \cong \underbrace{acb\cdots cba }_\text{length $n$} \underbrace{abc\cdots abc }_\text{length $n-1$} ,\]
which, after an application of Lemma~\ref{lem:towrite}, gives
\[  \underbrace{bca\cdots abc }_\text{length $n-2$}  a \underbrace{acb\cdots acb }_\text{length $n-1$} \cong \underbrace{cba\cdots cba }_\text{length $n-1$} \underbrace{abc\cdots abc }_\text{length $n-1$} \]
Evaluating the both sides of the above equation shows
\[  b \subset  \left(\bigoplus_{i = 1}^{n-1}  \underbrace{cbacba \cdots }_\text{length $i$} \otimes\underbrace{\cdots abcabc }_\text{length $i-1$}\right) \oplus \mathbf{1}.\]

Thus there exists an $i$ such that 
\[  \dim ( b \to \underbrace{cbacba \cdots }_\text{length $i$} \otimes\underbrace{\cdots abcabc }_\text{length $i-1$}) > 0.\]
Applying Frobenius reciprocity gives 
\[  \dim ( b \underbrace{ cbacba \cdots }_\text{length $i-1$} \to \underbrace{cbacba \cdots }_\text{length $i$})  > 0,\]
which, as $\underbrace{cbacba \cdots }_\text{length $i$}$ is simple, implies that 
\[   b \underbrace{ cbacba \cdots }_\text{length $i-1$} \cong \underbrace{cbacba \cdots }_\text{length $i$}.\]
However this is a contradiction to Lemma~\ref{lem:shift}.
\end{trivlist}

\item \textbf{Case: $n \equiv 2 \pmod 3$}

If $n \equiv 2 \pmod 3$, then the cyclic word $ \underbrace{bcabc\cdots }_\text{length n}$ ends in $c$. Consider the word
\[  \underbrace{bca\cdots cab }_\text{length $n-1$}  c \underbrace{bac\cdots acb }_\text{length $n-1$}.\]
If $n \equiv 2 \pmod 6$ then we can use the relations~\eqref{eq:zippers} to write 
\[  \underbrace{bca\cdots cab }_\text{length $n-1$}  c \underbrace{bac\cdots acb }_\text{length $n-1$} \cong \underbrace{cba\cdots bac }_\text{length $n-1$}  b \underbrace{cab\cdots abc }_\text{length $n-1$} ,\]
and if $n \equiv 5 \pmod 6$ then we can use the relations~\eqref{eq:zippers} to write 
\[ \underbrace{bca\cdots cab }_\text{length $n-1$}  c \underbrace{bac\cdots acb }_\text{length $n-1$} \cong  \underbrace{bac\cdots acb }_\text{length $n-1$} a  \underbrace{bca\cdots cab }_\text{length $n-1$}.\]
Hence we further break the proof up into two more sub-cases.

\begin{trivlist}\leftskip=4.5em
\item \textbf{Case: $n \equiv 2 \pmod 6$}

If $n \equiv 2 \pmod 6$ then we have
\[  \underbrace{bca\cdots cab }_\text{length $n-1$}  c \underbrace{bac\cdots acb }_\text{length $n-1$} \cong \underbrace{cba\cdots bac }_\text{length $n-1$}  b \underbrace{cab\cdots abc }_\text{length $n-1$} ,\]
to which we can apply relation~\eqref{eq:perm} on the first $n$ letters of the left hand side of the above equation to get
\[  \underbrace{abc\cdots bca }_\text{length $n-1$}  b \underbrace{bac\cdots acb }_\text{length $n-1$} \cong \underbrace{cba\cdots bac }_\text{length $n-1$}  b \underbrace{cab\cdots abc }_\text{length $n-1$} .\]
Evaluating the left hand side reveals that
\[ a \subset \underbrace{cba\cdots bac }_\text{length $n-1$}  b \underbrace{cab\cdots abc }_\text{length $n-1$},\]
from which we can use a similar argument as in the $n \equiv 3 \pmod 6$ case to see that
\[      \underbrace{acb\cdots bac }_\text{length $n$} \cong  \underbrace{cba\cdots acb}_\text{length $n$} .\] 
One can now proceed as in the $n \equiv 0 \pmod 6$ or $n \equiv 4 \pmod 6$ case to obtain a contradiction.

\vspace{1em}

\item \textbf{Case: $n \equiv 5 \pmod 6$}

If $n \equiv 5 \pmod 6$ then we have
\[  \underbrace{bca\cdots cab }_\text{length $n-1$}  c \underbrace{bac\cdots acb }_\text{length $n-1$} \cong \underbrace{bac\cdots acb }_\text{length $n-1$} a  \underbrace{bca\cdots cab }_\text{length $n-1$} ,\]
to which we can apply relation~\eqref{eq:perm} on the first $n$ letters of the left hand side of the above equation to get
\[  \underbrace{abc\cdots bca }_\text{length $n-1$}  b \underbrace{bac\cdots acb }_\text{length $n-1$} \cong \underbrace{bac\cdots acb }_\text{length $n-1$} a  \underbrace{bca\cdots cab }_\text{length $n-1$} .\]
Evaluating the left hand side reveals that
\[ a \subset \underbrace{bac\cdots acb }_\text{length $n-1$} a  \underbrace{bca\cdots cab }_\text{length $n-1$},\]
from which we can use a similar argument as in the $n \equiv 3 \pmod 6$ case to see that
\[     a \underbrace{bac\cdots acb }_\text{length $n-1$} \cong  \underbrace{bac\cdots cba}_\text{length $n$} .\] 
However this is a contradiction to Lemma~\ref{lem:shift}.
\end{trivlist}

\end{trivlist}

\end{proof}

\vspace{1em}

\changed{\begin{sublOne}
There are no rank preserving finite \changedd{unitarizable} quotients of $\operatorname{Fib}^{*3}$ such that only $a$ and $b$ commute with each other.
\end{sublOne}}

\begin{proof}
Consider such a quotient. As only $a$ and $b$ commute, we must have the other relations
\[   bcb \cong cbc \quad \text{ and } \quad aca \cong cac.\]

Let us consider the equation~\eqref{eq:perm}
\[\underbrace{abcab\cdots }_\text{length n}   \cong  \underbrace{bcabc\cdots }_\text{length n}.\]

If $n \geq 6$ then we have 
\begin{align*}
\underbrace{bcabca\cdots }_\text{length n} \cong& \underbrace{bcbaca\cdots }_\text{length n} \\
\cong& \underbrace{cbcaca\cdots }_\text{length n} \\
\cong& \underbrace{cbccac\cdots }_\text{length n} \\
\end{align*}
Thus 
\[\underbrace{abcabc\cdots }_\text{length n}   \cong  \underbrace{cbccac\cdots }_\text{length n}.\]
However we can expand the $cc$ on the right hand side to see that $\underbrace{abcabc\cdots }_\text{length n}$ is not simple, which is a contradiction.

If $n = 1,2,\cdots 5$ then one can quickly deduce a contradiction from equation~\eqref{eq:perm}. We work through the $n= 5$ case, and leave the other four cases to the reader, as they are all fairly similar.

If $n= 5$, then we have the relation
\[ abcab   \cong  bcaba.\]
We can write
\[ abcab \cong bacab \cong bcacb,\]
thus
\[  bcacb  \cong  bcaba.\]
Applying Lemma~\ref{lem:towrite} three times gives
\[    cb  \cong  ba.\]
Then, as $a$ and $b$ commute, we get
\[  cb \cong ab.\]
Finally, using Lemma~\ref{lem:towrite} gives $c \cong a$ which is a contradiction.

\vspace{1em}
\end{proof}

\changed{\begin{sublOne}
There are no rank preserving finite \changedd{unitarizable} quotients of $\operatorname{Fib}^{*3}$ such that only $b$ and $c$ commute with each other.
\end{sublOne}}
\begin{proof}

As only $b$ and $c$ commute, we must have the other relations
\[   aba \cong bab \quad \text{ and } \quad aca \cong cac.\]

Let us consider the equation~\eqref{eq:perm}
\[\underbrace{abcab\cdots }_\text{length n}   \cong  \underbrace{bcabc\cdots }_\text{length n}.\]

If $n \geq 7$ then we have 
\begin{align*}
 \underbrace{bcabcab\cdots }_\text{length n} \cong & \underbrace{bcacbab\cdots }_\text{length n} \\
\cong& \underbrace{bacabab\cdots }_\text{length n} \\
\cong & \underbrace{bacaaba\cdots }_\text{length n}.
\end{align*}
Thus 
\[\underbrace{abcabca\cdots }_\text{length n}   \cong  \underbrace{bacaaba\cdots }_\text{length n}.\]
However we can expand the $aa$ on the right hand side to see that $\underbrace{abcabca\cdots }_\text{length n}$ is not simple, which is a contradiction.

Now we work through the cases $n = 1,2,3,4,5,6$ and derive a contradiction in each case. Here the cases $n = 3$ and $n = 6$ are somewhat difficult, and we leave the easier cases to the reader

If $n = 3$ then we have the relation
\[  abc \cong bca.\]
We can left multiply by $b$ and right multiply by $c$ to get
\[  bab \oplus babc \cong cac \oplus bcac.\]
A direct computation shows $\dim( bab \to bab) = 1$, thus $bab$ is simple, and so either $bab \cong cac$ or $bab \subset bcac$. The former case we can rewrite as
\[   aba \cong aca,\]
to which we can apply Lemma~\ref{lem:towrite} to show $b = c$, a contradiction. The later case implies that
\[  \dim( bab \to bcac) > 0,\]
to which we can use Frobenius reciprocity to see
\[    \dim( bbab \to cac )>0\]
which implies either $bab \cong cac$ or $ab \subset cac$, each of which is a contradiction.

If $n = 6$ then we have the relation
\[  abcabc \cong bcabca.\]
We can write
\[  abcabc \cong abcacb \cong abacab \cong babcab,\]
thus 
\[  babcab \cong bcabca.\]
We apply Lemma~\ref{lem:towrite} to see
\[ abcab \cong cabca.\]
Right multiplying by $b$ gives
\[ abcab \oplus abca  \cong cabcab.\]
which contradicts the fact that cyclic words of length $6$ or less are simple.

\end{proof}

\vspace{1em}

\changed{\begin{sublOne}
There are no rank preserving finite \changedd{unitarizable} quotients of $\operatorname{Fib}^{*3}$ such that only $a$ and $c$ commute with each other.
\end{sublOne}}

\begin{proof}
Consider such a quotient. As only $a$ and $c$ commute, we must have the other relations
\[   aba \cong bab \quad \text{ and } \quad bcb \cong cbc.\]

Let us consider the equation~\eqref{eq:perm}
\[\underbrace{abcab\cdots }_\text{length n}   \cong  \underbrace{bcabc\cdots }_\text{length n}.\]

If $n \geq 6$ then we have 
\begin{align*}
 \underbrace{abcabc\cdots }_\text{length n} \cong & \underbrace{abacbc\cdots }_\text{length n} \\
\cong & \underbrace{babcbc\cdots }_\text{length n} \\
\cong & \underbrace{babbcb\cdots }_\text{length n}.
\end{align*}
Thus 
\[\underbrace{babbcb\cdots }_\text{length n}   \cong  \underbrace{bcabca\cdots }_\text{length n}.\]
However we can expand the $bb$ on the left hand side to see that $\underbrace{bcabca\cdots }_\text{length n}$ is not simple, which is a contradiction.

Now we work through the cases $n = 1,2,3,4,5$ and derive a contradiction in each case. Here the case $n = 2$ is difficult, and we leave the easier cases to the reader.

If $n = 2$ then we have the relation
\[   ab \cong bc.\]
Left multiplying by $b$, and right multiplying by $b$ gives
\[  babb \cong bbcb\]
which we expand to see
\[  ba \oplus bab \cong cb \oplus bcb.\]
As $ba \cong cb$ we must have $bab \cong bcb$, to which we can apply Lemma~\ref{lem:towrite} to show $a \cong c$, a contradiction.

\end{proof}

With the technical Lemma~\ref{lem:hard} in hand, we can now make a significant reduction to the problem of classifying finite \changedd{unitarizable} quotients of $\operatorname{Fib}^{*N}$. We show that any such quotient must factor through a certain Deligne product of $\operatorname{Fib}$ and $\mathcal{TT}_3$ categories.

\begin{lemma}
Let $\cC$ be a rank preserving finite \changedd{unitarizable} quotient of $\operatorname{Fib}^{*N}$, then $\cC$ is a \change{rank preserving \changedd{unitarizable}} quotient of 
\[  \operatorname{Fib}^{\boxtimes n} \boxtimes \mathcal{TT}_3^{\boxtimes m}\]
where $n + 2m =N$.
\end{lemma}
\begin{proof}
We prove by induction on $N$. 

\begin{trivlist}\leftskip=2em
\item $N= 1$:

The only quotient of $\operatorname{Fib}$ is $\operatorname{Fib}$ itself.

\vspace{1em}

\item $N=2$:

By \cite{MR3345186} the only finite \changedd{unitarizable} quotients of $\operatorname{Fib}* \operatorname{Fib}$ are $\operatorname{Fib}$, $\operatorname{Fib} \boxtimes \operatorname{Fib}$ and $ \mathcal{TT}_3$. Only the latter two of these are rank preserving quotients.

\vspace{1em}

\item Inductive step:

Suppose the result holds for $N-1$, and let $\cC$ be a finite rank preserving \changedd{unitarizable} quotient of $\operatorname{Fib}^{*N}$. Consider an arbitrary generator $\tau_1$ of $\cC$. We have two cases to consider, either $\tau_1$ commutes with all the other generators, or there exists another generator $\tau_2$ with which it doesn't commute.

If $\tau_1$ commutes with all the other generators, then $\cC$ must be a \changedd{unitarizable} quotient of $\operatorname{Fib} \boxtimes \mathcal{D}$, where $\mathcal{D}$ is a finite rank preserving \changedd{unitarizable} quotient of $\operatorname{Fib}^{*(N-1)}$. Thus we can use the inductive hypothesis to see that $\cC$ is a rank preserving \changedd{unitarizable} quotient of
\[ \operatorname{Fib}\boxtimes  \operatorname{Fib}^{\boxtimes n} \boxtimes \mathcal{TT}_3^{\boxtimes m}\]
where $n + 2m =N-1$.

If there exists a generator $\tau_2$ for which $\tau_1$ does not commute, then they must generate a $\mathcal{TT}_3$ subcategory of $\cC$. Let $\tau_3$ be a third generator of $\cC$, then the generators $\tau_1, \tau_2$, and $\tau_3$ generate a finite rank preserving \changedd{unitarizable} quotient of $\operatorname{Fib}^{*3}$. By Lemma~\ref{lem:hard} we must have that $\tau_3$ commutes with $\tau_1$ and $\tau_2$. As $\tau_3$ was chosen arbitrarily we have shown that $\tau_1$ and $\tau_2$ commute with all the other $N-2$ generators of $\operatorname{Fib}^{*N}$. Thus $\cC$ is a rank preserving \changedd{unitarizable} quotient of $\mathcal{TT}_3 \boxtimes \mathcal{D}$, where $\mathcal{D}$ is a finite \change{rank preserving \changedd{unitarizable}} quotient of $\operatorname{Fib}^{*(N-2)}$. Therefore we can use the inductive hypothesis to see that $\cC$ is a rank preserving \changedd{unitarizable} quotient of
\[   \mathcal{TT}_3\boxtimes \operatorname{Fib}^{\boxtimes n} \boxtimes \mathcal{TT}_3^{\boxtimes m}\]
where $n + 2m =N-2$.

Irrespective on whether $\tau_1$ commutes with all the other generators or not, we have shown that $\cC$ is a rank preserving \changedd{unitarizable} quotient of 
\[   \operatorname{Fib}^{\boxtimes n} \boxtimes \mathcal{TT}_3^{\boxtimes m}\]
where $n + 2m =N$.
\end{trivlist}
\end{proof}

Recall that we are interested in classifying finite cyclic \changedd{unitarizable} quotients of $\operatorname{Fib}^{*N}$. So far our lemmas have just required the quotient be finite. While we could continue without a symmetry condition, and obtain a classification of all finite \changedd{unitarizable} quotients of $\operatorname{Fib}^{*N}$, such a result is beyond the scope of this paper, and the proof of such a generalised result appears tedious, if straightforward. In the following Remark we discuss which of the categories 
\[   \operatorname{Fib}^{\boxtimes n} \boxtimes \mathcal{TT}_3^{\boxtimes m}\]
are cyclic, and thus could have cyclic quotients.

\begin{rmk}\label{rmk:syms}
Recall a finite cyclic \changedd{unitarizable} quotient of $\operatorname{Fib}^{*N}$ is a finite quotient that has an action of $\Z{N}$ such that the generator of $\Z{N}$ maps
\[  \tau_1 \mapsto \tau_2 \mapsto \cdots \mapsto \tau_N \mapsto \tau_1.\]
Clearly the finite quotient $\operatorname{Fib}^{\boxtimes N}$ is cyclic, with the symmetry coming from the canonical $\Z{N}$ action permuting the factors.

When $N$ is even, we also have the quotient $\mathcal{TT}_3^{\boxtimes \frac{N}{2}}$ is cyclic. In this case the symmetry is more involved. Let $\tau$ and $\phi$ be the Fibonacci generators of $\mathcal{TT}_3$, then the generator of $\Z{N}$ acts on $\mathcal{TT}_3^{\boxtimes \frac{N}{2}}$ by 
\[         \tau_1 \mapsto \tau_2 \mapsto \cdots \tau_{ \frac{N}{2}} \mapsto \phi_1 \mapsto \phi_2 \mapsto \cdots \mapsto \phi_{ \frac{N}{2}} \mapsto \tau_1.\]

As both the categories $\operatorname{Fib}^{\boxtimes N}$  and $\mathcal{TT}_3^{\boxtimes \frac{N}{2}}$ have trivial universal grading group, and have no non-trivial invertible objects, we have from \cite[Theorem 3.4]{MR3354332} along with \cite[Theorem 5.5]{MR2763944} that both of the above $\Z{N}$ actions lift to categorical actions of $\Z{N}$ on the quotients.
\end{rmk}

It is clear by the definition of a cyclic quotient that the categories mentioned in the above Remark are the only possible cyclic \changedd{unitarizable} quotients of $\operatorname{Fib}^{*N}$. Thus we have the following 
\begin{cor}
Let $\cC$ be a rank preserving finite cyclic \changedd{unitarizable} quotient of $\operatorname{Fib}^{*N}$\changed{. T}hen $\cC$ is a cyclic \changedd{unitarizable} quotient of either
\[  \operatorname{Fib}^{\boxtimes N},\]
or, if $N$ is even, 
\[   \mathcal{TT}_3^{\boxtimes \frac{N}{2}}.\]
\end{cor}

Finally we need to determine if there are any further \changedd{unitarizable} quotients of $\operatorname{Fib}^{\boxtimes N}$ or $\mathcal{TT}_3^{\boxtimes N}$. This is equivalent to classifying central commutative algebras in these categories, a hard problem in general \cite{MR2909758}. However we are able to find arguments specific to the two above categories that bypass having to classify such algebras. The idea behind these arguments is to simply compute all the simple objects in such a quotient. We find that the simple objects of $\operatorname{Fib}^{\boxtimes N}$ or $\mathcal{TT}_3^{\boxtimes N}$ remain simple and distinct in any quotient. Thus only the trivial quotients exist. This argument was inspired by a similar technique used in \cite{MR3536926}.

We begin with the $\operatorname{Fib}^{\boxtimes N}$ case.
\begin{lemma}\label{lem:fibbox}
Let $\cC$ be a rank preserving \changedd{unitarizable} quotient of $\operatorname{Fib}^{\boxtimes N}$, then $\cC$ is monoidally equivalent to $\operatorname{Fib}^{\boxtimes N}$.
\end{lemma}
\begin{proof}
As $\cC$ is a rank preserving finite \changedd{unitarizable} quotient of $\operatorname{Fib}^{\boxtimes N}$, it contains the commuting distinct simple objects $\tau_1, \tau_2, \cdots, \tau_N$ each individually satisfying the Fibonacci fusion rule. 

Consider the words of $\cC$ that contain at most one of each of the objects $\tau_1, \tau_2, \cdots, \tau_N$. We will call such words \textit{source-simple} words. We say two source-simple words are identical if they share the same letter set.

To prove the claim of this lemma, we will show that all source-simple words are simple, and two source-simple words are distinct if they are non-identical. With this result, we count the global dimension of $\cC$ as at least $\left( \frac{5 + \sqrt{5}}{2}\right)^N$, hence $\cC$ must be the trivial quotient of $\operatorname{Fib}^{\boxtimes N}$.

We begin by showing all source-simple words are simple.

We induct on the length of the word.

As the objects $\tau_1, \tau_2, \cdots, \tau_N$ are simple, we have that the result holds for source-simple words of length $1$.

Suppose the result holds for source-simple words of length $k-1$ for $k \leq N$. That is, suppose all source-simple words of length $k-1$ or less are simple. Consider a source-simple word of length $k$. As the labels for the $N$ generators of $\operatorname{Fib}^{\boxtimes N}$ were chosen arbitrarily we can assume that this words begins with $\tau_1$ and ends with $\tau_2$. We compute the endomorphisms of this word.
\begin{align*}
\dim ( \underbrace{\tau_1 \cdots  \tau_2}_\text{length $k$} \to \underbrace{\tau_1 \cdots  \tau_2}_\text{length $k$}) = & \dim ( \underbrace{\tau_1 \tau_1 \cdots  }_\text{length $k$}   \to \underbrace{ \cdots  \tau_2 \tau_2}_\text{length $k$}) \\
									=&  \dim ( \underbrace{\tau_1 \cdots }_\text{length $k-1$}\oplus \underbrace{\cdots}_\text{length $k-2$}   \to \underbrace{ \cdots  \tau_2}_\text{length $k-1$} \oplus \underbrace{\cdots}_\text{length $k-2$}) \\
									=&  \dim ( \underbrace{\tau_1 \cdots }_\text{length $k-1$}  \to \underbrace{ \cdots  \tau_2}_\text{length $k-1$} ) + \dim ( \underbrace{ \cdots }_\text{length $k-2$}  \to \underbrace{ \cdots  \tau_2}_\text{length $k-1$} ) \\
									&+ \dim ( \underbrace{\tau_1 \cdots }_\text{length $k-1$}  \to \underbrace{ \cdots}_\text{length $k-2$} ) + \dim ( \underbrace{ \cdots }_\text{length $k-2$}  \to \underbrace{ \cdots }_\text{length $k-2$} )  \\
									=&1 +  \dim ( \underbrace{\tau_1 \cdots }_\text{length $k-1$}  \to \underbrace{ \cdots  \tau_2}_\text{length $k-1$} ).
\end{align*}
Here the last step uses the inductive hypothesis. 

As both $\underbrace{\tau_1 \cdots }_\text{length $k-1$}$ and $\underbrace{ \cdots  \tau_2}_\text{length $k-1$}$ are source-simple words, we have that they are simple by the inductive hypothesis. Thus
\[ \dim ( \underbrace{\tau_1 \cdots }_\text{length $k-1$}  \to \underbrace{ \cdots  \tau_2}_\text{length $k-1$} )\]
is either equal to $0$ or $1$.

If 
\[ \dim ( \underbrace{\tau_1 \cdots }_\text{length $k-1$}  \to \underbrace{ \cdots  \tau_2}_\text{length $k-1$} ) =1 \]
then we have the equality $\underbrace{\tau_1 \cdots }_\text{length $k-1$} \cong \underbrace{ \cdots  \tau_2}_\text{length $k-1$}$ as both words are simple. Using the commutativity of the $\tau_i$'s we can write
\[  \underbrace{\tau_1 \cdots }_\text{length $k-1$} \cong \underbrace{\tau_2 \cdots  }_\text{length $k-1$}\]
to which we can apply Lemma~\ref{lem:towrite} $k-1$ times (as the rightmost $k-1$ letters of both words are identical up to order) to see $\tau_1 \cong \tau_2$. Thus we have a contradiction, so 
\[ \dim ( \underbrace{\tau_1 \cdots }_\text{length $k-1$}  \to \underbrace{ \cdots  \tau_2}_\text{length $k-1$} ) =0. \]
Therefore the endomorphism space of $ \underbrace{\tau_1 \cdots  \tau_2}_\text{length $k$}$ is 1 dimensional, and hence is simple. Thus, all source-simple words of length $k$ are simple, completing the \change{induction}.

\vspace{2em}

We now prove that all non-identical source-simple words are distinct. Again we induct on the length of the word.

As the objects $\tau_1, \tau_2, \cdots, \tau_N$ are distinct, we have that the result holds for source-simple words of length $1$.

Suppose the result holds for source-simple words of length $k-1$ for $k \leq N$. That is, suppose all non-identical source-simple words of length $k-1$ are distinct. Consider two source-simple words, and suppose they are equal. We consider two cases, either these two words share a letter, or they do not share a letter.

If the two words do not share a letter then consider the word obtained by concatenating these two words together. As the two component words share no letters, we have the concatenated word is source-simple, and thus simple by the earlier result. However as the two component words are equal, we can use Frobenius reciprocity to show that there is a non-trivial morphism from the tensor unit to the concatenated word, a contradiction. Thus this case can not occur.

Now suppose the two words share a letter, say $\tau_1$. We can use the commutativity of the $\tau_i$'s to bring the shared $\tau_1$ to the left hand side of each word. Then we can apply Lemma~\ref{lem:towrite} to see that the two words of length $k-1$ obtained by removing $\tau_1$ are equal. The inductive hypothesis then gives that these two words of length $k-1$ are identical. As the two words of length $k$ are obtained from this length $k-1$ word by appending a $\tau_1$, we thus have these two source-simple words of length $k$ are identical, completing the induction.
\end{proof}

Next we consider the $\mathcal{TT}_3^{\boxtimes N}$ case. The arguments here are slight alterations of the $\operatorname{Fib}^{\boxtimes N}$ case.

\begin{lemma}\label{lem:TTbox}
Let $\cC$ be a rank preserving \changedd{unitarizable} quotient of $\mathcal{TT}_3^{\boxtimes N}$. Then $\cC$ is monoidally equivalent to $\mathcal{TT}_3^{\boxtimes N}$.
\end{lemma}
\begin{proof}
As $\cC$ is a rank preserving finite quotient of $\mathcal{TT}_3^{\boxtimes N}$, it contains the distinct simple objects $\tau_1, \phi_1,  \tau_2,\phi_2  \cdots, \tau_N, \phi_N$ each individually satisfying the $\operatorname{Fib}$ fusion rules, and 
\[   \tau_i \phi_i \tau_i \cong \phi_i\tau_i\phi_i, \qquad       \tau_i \tau_j  \cong \tau_j \tau_i,    \qquad \phi_i \tau_j  \cong \tau_j \phi_i,   \qquad   \phi_i \phi_j  \cong \phi_j \phi_i,\]
for $i \neq j$.

Consider the words of $\cC$ that for each $i$, contain either at most one $\tau_i$ and at most one $\phi_i$, or at most one of the words $\tau_i\phi_i\tau_i$. We will call such words \textit{source-simple} words. We say two source-simple words are identical if for each $i$, the two subwords consisting of $\tau_i$'s and $\phi_i$'s are equal in $\mathcal{TT}_3$. 

To prove the claim of this lemma, we will show that all source-simple words are simple, and two source-simple words are distinct if they are non-identical. With this result, we count the global dimension of $\cC$ as at least $\left(20 + 8\sqrt{5}\right)^N$, hence $\cC$ must be the trivial quotient of $\mathcal{TT}_3^{\boxtimes N}$.

As in the $\operatorname{Fib}^{\boxtimes N}$ case we begin by proving first proving that all source-simple words are simple. We proceed by induction.

As the objects $\tau_1, \phi_1,  \tau_2,\phi_2,  \cdots, \tau_N, \phi_N$ are all simple, we have that the result holds for source-simple words of length $1$.

Suppose the result holds for words of length $k-1$ for $k < N$. That is, suppose all source-simple words of length $k-1$ are simple. Consider a source-simple word of length $k$. We can assume that this word begins with the letter $\tau_1$. We have four cases to consider, either the word ends in $\tau_1$, $\tau_2$, $\phi_1$, or $\phi_2$.

\begin{trivlist}\leftskip=2em
\item \textbf{Case: The source-simple word ends in $\tau_1$}

If the source-simple word begins and ends in $\tau_1$ then the word must also contain a $\phi_1$ and no additional $\tau_1$'s, by the definition of source-simple. Using the inductive hypothesis we compute the endomorphism space of this word
\[  \dim ( \underbrace{\tau_1 \cdots  \tau_1}_\text{length $k$} \to \underbrace{\tau_1 \cdots  \tau_1}_\text{length $k$}) = 1 + \dim ( \underbrace{\tau_1 \cdots }_\text{length $k-1$} \to \underbrace{ \cdots \tau_1 }_\text{length $k-1$}).\]

As both $\underbrace{\tau_1 \cdots }_\text{length $k-1$}$ and $ \underbrace{ \cdots \tau_1 }_\text{length $k-1$}$ are source-simple, they are simple by the inductive hypothesis. Thus 
\[ \dim ( \underbrace{\tau_1 \cdots }_\text{length $k-1$} \to \underbrace{ \cdots \tau_1 }_\text{length $k-1$})\]
is equal to either $1$ or $0$.

If 
\[ \dim ( \underbrace{\tau_1 \cdots }_\text{length $k-1$} \to \underbrace{ \cdots \tau_1 }_\text{length $k-1$}) = 1\]
then 
\[ \underbrace{\tau_1 \cdots }_\text{length $k-1$}\cong \underbrace{ \cdots \tau_1 }_\text{length $k-1$}.\]
As the parent length $k$ word only contained two $\tau_1$'s and a single $\phi_1$ we have that each of these words contains a $\tau_1 \phi_1$ and $\phi_1 \tau_1$ subword respectively, and no other $\tau_1$'s or $\phi_1$'s. Thus we can commute to the right to get
\[ \underbrace{\cdots \tau_1 \phi_1}_\text{length $k-1$}\cong \underbrace{ \cdots \phi_1 \tau_1 }_\text{length $k-1$}.\]
The leftmost $k-2$ letters of both these words are the same, so we can apply Lemma~\ref{lem:towrite} $k-2$ times to get $\tau_1\phi_1 \cong \phi_1 \tau_1$, a contradiction. Thus
\[ \dim ( \underbrace{\tau_1 \cdots }_\text{length $k-1$} \to \underbrace{ \cdots \tau_1 }_\text{length $k-1$}) = 0\] 
and so $\underbrace{\tau_1 \cdots  \tau_1}_\text{length $k$}$ is simple.

\vspace{1em}

\item \textbf{Case: The source-simple word ends in $\tau_2$}

Suppose we have a source-simple word that begins with $\tau_1$ and ends in $\tau_2$. Using the inductive hypothesis we compute the endomorphism space of this word
\[  \dim ( \underbrace{\tau_1 \cdots  \tau_2}_\text{length $k$} \to \underbrace{\tau_1 \cdots  \tau_2}_\text{length $k$}) = 1 + \dim ( \underbrace{\tau_1 \cdots }_\text{length $k-1$} \to \underbrace{ \cdots \tau_2 }_\text{length $k-1$}).\]

As both $\underbrace{\tau_1 \cdots }_\text{length $k-1$}$ and $ \underbrace{ \cdots \tau_2 }_\text{length $k-1$}$ are source-simple, they are simple. Thus 
\[ \dim ( \underbrace{\tau_1 \cdots }_\text{length $k-1$} \to \underbrace{ \cdots \tau_2 }_\text{length $k-1$})\]
is equal to either $1$ or $0$.

If 
\[ \dim ( \underbrace{\tau_1 \cdots }_\text{length $k-1$} \to \underbrace{ \cdots \tau_2 }_\text{length $k-1$}) = 1\]
then 
\[ \underbrace{\tau_1 \cdots }_\text{length $k-1$}\cong \underbrace{ \cdots \tau_2 }_\text{length $k-1$}.\]

We have three further sub-cases to consider, either the source-simple word of length $k$ contains only a single $\tau_1$, and no additional $\tau_1$'s or $\phi_1$'s, or the source-simple word of length $k$ contains a $\tau_1 \phi_1$, and no additional $\tau_1$'s or $\phi_1$'s,
 or the source-simple word of length $k$ contains a $\tau_1 \phi_1 \tau_1$, and no additional $\tau_1$'s or $\phi_1$'s.
 
 \begin{trivlist}\leftskip=4em
\item \textbf{Sub-case: The source-simple word of length $k$ contains only a single $\tau_1$, and no additional $\tau_1$'s or $\phi_1$'s}

In this case we can commute the $\tau_1$ to the far right of $\underbrace{\tau_1 \cdots }_\text{length $k-1$}$, and cancel the shared leftmost $k-2$ letters in the equation
\[ \underbrace{\cdots \tau_1 }_\text{length $k-1$}\cong \underbrace{ \cdots \tau_2 }_\text{length $k-1$},\]
to obtain $\tau_1 \cong \tau_2$, a contradiction.

\vspace{1em}

\item \textbf{Sub-case: The source-simple word of length $k$ contains a $\tau_1 \phi_1$, and no additional $\tau_1$'s or $\phi_1$'s}

In this case we can commute the $\phi_1$ to the far right of the word $\underbrace{ \cdots \tau_2 }_\text{length $k-1$}$, and we can commute the $\tau_1$ and $\phi_1$ to the far right of the word $ \underbrace{\tau_1 \cdots }_\text{length $k-1$}$ to get the equality
\[ \underbrace{ \cdots \tau_1 \phi_1 }_\text{length $k-1$}\cong \underbrace{ \cdots \tau_2 \phi_1 }_\text{length $k-1$}.\]
Cancelling the shared leftmost $k-3$ letters in the above equation, and the shared rightmost letter gives $\tau_1  \cong \tau_2 $, a contradiction.

\vspace{1em}

\item \textbf{Sub-case: The source-simple word of length $k$  contains a $\tau_1 \phi_1 \tau_1$, and no additional $\tau_1$'s or $\phi_1$'s}

In this case we can commute the $\phi_1 \tau_1$ to the far right of the word $\underbrace{ \cdots \tau_2 }_\text{length $k-1$}$, and we can commute the $\tau_1 \phi_1 \tau_1$ to the far right of the word $ \underbrace{\tau_1 \cdots }_\text{length $k-1$}$ to get the equality
\[ \underbrace{ \cdots \tau_1 \phi_1 \tau_1 }_\text{length $k-1$}\cong \underbrace{ \cdots \tau_2 \phi_1 \tau_1 }_\text{length $k-1$}.\]
Cancelling the shared leftmost $k-4$ letters in the above equation, and the two shared rightmost letter gives $\tau_1  \cong \tau_2 $, a contradiction.

\end{trivlist}

In all cases we get a contradiction, thus
\[ \dim ( \underbrace{\tau_1 \cdots }_\text{length $k-1$} \to \underbrace{ \cdots \tau_2 }_\text{length $k-1$}) = 0.\]
Hence $\underbrace{\tau_1 \cdots  \tau_2}_\text{length $k$}$ is simple.

\vspace{1em}

\item \textbf{Case: The source-simple word ends in $\phi_1$}

If the source-simple word begins with $\tau_1$ and ends in $\phi_1$ then the word must contain no additional $\tau_1$'s or $\phi_1$'s, by the definition of source-simple. Using the inductive hypothesis we compute the endomorphism space of this word
\[  \dim ( \underbrace{\tau_1 \cdots  \phi_1}_\text{length $k$} \to \underbrace{\tau_1 \cdots  \phi_1}_\text{length $k$}) = 1 + \dim ( \underbrace{\tau_1 \cdots }_\text{length $k-1$} \to \underbrace{ \cdots \phi_1 }_\text{length $k-1$}).\]

As both $\underbrace{\tau_1 \cdots }_\text{length $k-1$}$ and $ \underbrace{ \cdots \phi_1 }_\text{length $k-1$})$ are source-simple, they are simple. Thus 
\[ \dim ( \underbrace{\tau_1 \cdots }_\text{length $k-1$} \to \underbrace{ \cdots \phi_1 }_\text{length $k-1$})\]
is equal to either $1$ or $0$.

If 
\[ \dim ( \underbrace{\tau_1 \cdots }_\text{length $k-1$} \to \underbrace{ \cdots \phi_1 }_\text{length $k-1$}) = 1\]
then 
\[ \underbrace{\tau_1 \cdots }_\text{length $k-1$}\cong \underbrace{ \cdots \phi_1 }_\text{length $k-1$}.\]
As both these words contain no additional $\tau_1$'s or $\phi_1$'s we can commute the $\tau_1$ to the far right of the left word to get
\[ \underbrace{\cdots \tau_1 }_\text{length $k-1$}\cong \underbrace{ \cdots \phi_1 }_\text{length $k-1$},\]
with the leftmost $k-2$ letters of both words being the same. Applying Lemma~\ref{lem:towrite} $k-2$ times gives $\tau_1 \cong \phi_1$, a contradiction. Thus 
\[ \dim ( \underbrace{\tau_1 \cdots }_\text{length $k-1$} \to \underbrace{ \cdots \phi_1 }_\text{length $k-1$}) = 0\] 
and so $\underbrace{\tau_1 \cdots  \phi_1}_\text{length $k$}$ is simple.

\vspace{1em}

\item \textbf{Case: The source-simple word ends in $\phi_2$}

This case is near identical to the case where the word ends in $\tau_2$.

\vspace{1em}

\end{trivlist}

This case by case analysis shows all source-simple words of length $k$ are simple, completing the induction.

\vspace{2em}

We now prove that all non-identical source-simple words are distinct. Again we induct on the length of the word.

As the objects $\tau_1, \phi_1,  \tau_2,\phi_2  \cdots, \tau_N, \phi_N$ are all distinct, we have that the result holds for source-simple words of length $1$.

Suppose the result holds for source-simple words of length $k-1$ for $k \leq N$. That is, suppose all non-identical source-simple words of length $k-1$ are distinct. Consider two equal source-simple words of length $k-1$. We have two cases to consider, either both words share a letter, or they don't share a letter.

Suppose the two words don't share a letter. Then the concatenation of these two words is another source-simple word, and is thus simple. However as the two component words are equal, we can use Frobenius reciprocity to show that the concatenated word has a non-trivial morphism to the tensor unit. Thus we have a contradiction, so this case can not occur.

Suppose the two words share a letter $\tau_1$. Considering all possible positions of the $\tau_1$'s and $\phi_1$'s in both words, in all but one bad case we can simultaneously commute (by the standard commuting relations, and with the relation $\phi_1 \tau_1 \phi_1= \tau_1 \phi_1 \tau_1$) either a $\tau_1$ or a $\phi_1$ to the same side of both words and then cancel with Lemma~\ref{lem:towrite} and appeal to the inductive hypothesis. The bad case is if one word contains just a $\phi_1 \tau_1$ and the other contains just a $\tau_1 \phi_1$. For this bad case we can commute these blocks to the far left of both words to get an equality of the form
\[   \phi_1 \tau_1 \underbrace{\cdots}_\text{length $k-2$} \cong  \tau_1 \phi_1\underbrace{\cdots}_\text{length $k-2$} .\]
We thus have
\begin{align*} 1 =  \dim ( \phi_1 \tau_1 \underbrace{\cdots}_\text{length $k-2$} \to  \tau_1 \phi_1\underbrace{\cdots}_\text{length $k-2$}) =& \dim ( \phi_1 \tau_1 \phi_1 \tau_1 \underbrace{\cdots}_\text{length $k-2$} \to  \underbrace{\cdots}_\text{length $k-2$})\\
																					    =& \dim ( \phi_1 \phi_1 \tau_1 \phi_1 \underbrace{\cdots}_\text{length $k-2$} \to  \underbrace{\cdots}_\text{length $k-2$})\\
																					    =& \dim (  \phi_1 \tau_1 \phi_1 \underbrace{\cdots}_\text{length $k-2$} \to \underbrace{\cdots}_\text{length $k-2$})  \\
																					    &  + \dim (  \tau_1 \phi_1 \underbrace{\cdots}_\text{length $k-2$} \to \underbrace{\cdots}_\text{length $k-2$}) \\
																					 =& \dim (   \tau_1 \phi_1 \underbrace{\cdots}_\text{length $k-2$} \to\phi_1 \underbrace{\cdots}_\text{length $k-2$})  \\
																					    &  + \dim (   \phi_1 \underbrace{\cdots}_\text{length $k-2$} \to \tau_1\underbrace{\cdots}_\text{length $k-2$}) \\
																					    =& 0.												    
  \end{align*}
Here the last step follows from the fact that the words $ \tau_1 \phi_1 \underbrace{\cdots}_\text{length $k-2$}$, $ \phi_1 \underbrace{\cdots}_\text{length $k-2$}$, and $\tau_1\underbrace{\cdots}_\text{length $k-2$}$ are all source-simple. Thus 
\[ \dim (   \tau_1 \phi_1 \underbrace{\cdots}_\text{length $k-2$} \to\phi_1 \underbrace{\cdots}_\text{length $k-2$})=0\]
because it is a morphism space between simples of different dimension, and
\[ \dim (   \phi_1 \underbrace{\cdots}_\text{length $k-2$} \to \tau_1\underbrace{\cdots}_\text{length $k-2$}) = 0 \]
by the inductive hypothesis. Thus we have a contradiction, thus this single bad case can not occur. Hence all non-identical source-simple words of length $k$ are distinct, completing the induction.
\end{proof}

Summing up the results of this section, we have Theorem~\ref{thm:fin}.

\section{The Classification Result}\label{sec:class}
Seeing that we have classified finite rank preserving cyclic \changedd{unitarizable} quotients of $\operatorname{Fib}^{*N}$ in the previous section, we can now apply Theorem~\ref{thm:semi} to give a classification of \change{unitary} fusion categories $\otimes$-generated by an object of dimension $\frac{1 + \sqrt{5}}{2}$.

\begin{proof}[Proof of Theorem~\ref{thm:main}]
Let $\cC$ be a \change{unitary} fusion category $\otimes$-generated by an object of dimension $\frac{1 +\sqrt{5}}{2}$. Then \changedd{ in particular $\cC$ is unitarizable, so} by Theorem~\ref{thm:semi} we have that $\cC$ is monoidally equivalent to a semi-direct product of a finite rank preserving cyclic \changedd{unitarizable} quotient of $\operatorname{Fib}^{*N}$ by a cyclic group that factors through the cyclic $\Z{N}$ action. From Theorem~\ref{thm:fin} we know that the finite cyclic \changedd{unitarizable} quotients of $\operatorname{Fib}^{*N}$ are $\operatorname{Fib}^{\boxtimes N}$ for any $N$, and $ \mathcal{TT}_3^{\boxtimes \frac{N}{2}}$ for $N$ even, with the cyclic actions given by Remark~\ref{rmk:syms}. For a cyclic group to factor through the cyclic $\Z{N}$ action, the order must be a multiple of $N$. Hence $\cC$ is monoidally equivalent to either
\begin{itemize}
\item A crossed product of $\operatorname{Fib}^{\boxtimes N}$ by $\Z{NM}$, with the generator of $\Z{NM}$ acting by cyclically permuting the factors, or
\item A crossed product of $ \mathcal{TT}_3^{\boxtimes \frac{N}{2}}$ by $\Z{NM}$, with the group $\Z{NM}$ acting by factoring through the $\Z{N}$ action described in Remark~\ref{rmk:syms}.
\end{itemize}
Thus $\cC$ is monoidally equivalent to either 
\begin{itemize}
\item $ \operatorname{Fib}^{\boxtimes N} \overset{\omega}{\rtimes} \Z{NM} \text{ where } N,M\in \mathbb{N}\text{ and } \omega \in H^3(\Z{NM} , \mathbb{C}^\times) , \text{ or }$
\item $ \mathcal{TT}_{3}^{\boxtimes N} \overset{\omega}{\rtimes} \Z{2NM} \text{ where } N,M\in \mathbb{N}\text{ and } \omega \in H^3(\Z{2NM} , \mathbb{C}^\times).$
\end{itemize}
\changedd{Each of these categories has a unitary structure, and further, this unitary structure is unique by \cite{1906.09710}. Finally \cite[Theorem 1]{1906.09710} shows that the monoidal equivalence above is naturally isomorphic to a unitary monoidal equivalence.}
\end{proof}

\footnotesize

\newcommand{\noopsort}[1]{}\def\cprime{$'$} \def\cprime{$'$} \def\cprime{$'$}

  Cain Edie-Michell, \textsc{Department of Mathematics, Vanderbilt University,
   Nashville, USA}\par\nopagebreak
  \textit{E-mail address}, \texttt{cain.edie-michell@vanderbilt.edu}

\end{document}